\documentclass[tbtags,reqno]{amsart}

\usepackage{geometry}
\usepackage{mathrsfs}
\usepackage{latexsym}
\usepackage{amsthm,amsopn,tabmac,amsfonts,amssymb,epsfig,color,pstricks,pb-diagram,mathdots,stmaryrd}
\usepackage[active]{srcltx}
\numberwithin{equation}{section}
\makeatletter
\renewcommand{\subsubsection}{\@startsection
{subsubsection}
{3}
{0mm}
{\baselineskip}
{-0.5\baselineskip}
{\normalfont\normalsize\bfseries}}
\makeatother

\newtheorem{theorem}{Theorem}
\newtheorem{lemma}[theorem]{Lemma}
\newtheorem{proposition}[theorem]{Proposition}
\newtheorem{example}[theorem]{Example}
\newtheorem{conjecture}[theorem]{Conjecture}

\newtheorem{corollary}[theorem]{Corollary}

\newtheorem{definition}[theorem]{Definition}

\newtheorem{remark}[theorem]{Remark}
\newtheorem*{acknow}{Acknowledgments}

\def\cal L{{\mathcal L}}

\newcommand{\cercle}[1]{\ensuremath{\setlength{\unitlength}{1ex}\begin{picture}(2.8,1.5)\put(1.4,0.7){\circle{2.8}\makebox(-5.6,0){#1}}\end{picture}}}

\def\beq{\begin{equation}}
\def\eeq{\end{equation}}
\def\bea{\begin{align}}
\def\eea{\end{align}}

\title{
$m$-Symmetric functions, non-symmetric Macdonald polynomials and positivity conjectures}
\begin{document}

\author{L. Lapointe}
\address{Instituto de Matem\'aticas, Universidad de
Talca, 2 norte 685, Talca, Chile.}
\email{llapointe@utalca.cl }

\begin{abstract}
  We study the space, $R_m$, of $m$-symmetric functions consisting of polynomials that are symmetric in the variables $x_{m+1},x_{m+2},x_{m+3},\dots$ but have no special symmetry in the variables $x_1,\dots,x_m$.  We obtain $m$-symmetric Macdonald polynomials by $t$-symmetrizing non-symmetric Macdonald polynomials, and show that they form a basis of $R_m$.  
  We define $m$-symmetric Schur functions through a somewhat complicated process involving their dual basis, tableaux combinatorics, and the Hecke algebra generators, and then prove some of their most elementary properties.  We conjecture that the $m$-symmetric Macdonald polynomials (suitably normalized and plethystically modified) expand positively in terms of $m$-symmetric Schur functions. We obtain  relations on the $(q,t)$-Koska coefficients $K_{\Omega \Lambda}(q,t)$ in the $m$-symmetric world, and show in particular that the usual $(q,t)$-Koska coefficients are special cases of the  $K_{\Omega \Lambda}(q,t)$'s.  Finally, we show that when $m$ is large, the positivity conjecture, modulo a certain subspace, becomes a positivity conjecture on the expansion of non-symmetric Macdonald polynomials in terms of non-symmetric Hall-Littlewood polynomials.  
\end{abstract}


\thanks{Funding: this work was supported by
the Fondo Nacional de Desarrollo Cient\'{\i}fico y Tecnol\'ogico de Chile (FONDECYT) Regular Grant \#1210688.}

\maketitle

\section{Introduction}	

Macdonald polynomials are a family of symmetric functions depending on two parameters $q$ and $t$ \cite{M}. It was shown in \cite{GH,Haiman} that the Schur expansion of the Macdonald polynomials (suitably normalized and plethystically modified) is positive, that is, that 
$$
\tilde J_\lambda(x;q,t) = \sum_\mu K_{\mu \lambda}(q,t) \, s_\mu (x) \qquad {\rm with~} K_{\mu \lambda}(q,t) \in \mathbb N[q,t]
$$
where $\tilde J_\lambda(x;q,t)$ stands for the plethystically modified Macdonald polynomial indexed by the partition $\lambda$.

There exists a non-symmetric version $E_\eta(x;q,t)$ of the Macdonald polynomials \cite{Che,Mac2} which are indexed by compositions and form a basis of the ring of polynomials in a given number  $N$ of   variables.
Even though  many combinatorial properties of the non-symmetric Macdonald polynomials are similar in nature to those of the Macdonald polynomials,
there was until now no known Macdonald positivity phenomenon in the non-symmetric case (except in special cases  \cite{AS,A1}).
One of the main goals of this article is to develop the correct framework in which to extend the original Macdonald positivity conjectures to the non-symmetric case.  In order to describe this framework, which is that of the $m$-symmetric functions,  we first  need to make a detour through superspace.

Generalizations to superspace of the original Macdonald positivity conjecture were provided in \cite{BDLM,BDLM2}.  In this setting, the Macdonald polynomials, denoted $J_\Lambda^{({\rm SUSY})}(q,t)$ and indexed by superpartitions, essentially correspond to non-symmetric Macdonald polynomials whose variables $x_1,\dots,x_m$ (resp. $x_{m+1},\dots,x_N$) are antisymmetrized (resp. $t$-symmetrized), with $m$ being the fermionic sector. In short,
\begin{equation} \label{eqanti}
  J_\Lambda^{({\rm SUSY})}(q,t) \longleftrightarrow \mathcal A_{1,\dots, m} \mathcal S^t_{m+1, N} E_\eta(x;q,t)
  \end{equation}
where $E_\eta(x;q,t)$ is a non-symmetric Macdonald polynomial, and where 
$\mathcal A_{1,\dots, m}$ (resp. $\mathcal S^t_{m+1, N}$) stands for the antisymmetrization (resp. $t$-symmetrization) operator.  The extension of the
 original Macdonald positivity conjecture is then that
\begin{equation} \label{susypos}
\tilde J_\Lambda^{({\rm SUSY})}(q,t) = \sum_\Omega K_{\Omega \Lambda}(q,t) \, s_\Omega^{({\rm SUSY})}  \qquad {\rm with~} K_{\Omega \Lambda}(q,t) \in \mathbb N[q,t]
\end{equation}
where $s_\Omega^{({\rm SUSY})}=J_\Lambda^{({\rm SUSY})}(0,0)$, and where
$\tilde J_\Lambda^{({\rm SUSY})}(q,t)$ is a plethystically modified version of $J_\Lambda^{({\rm SUSY})}(q,t)$.

For quite a while we have surmised that applying the antisymmetrization operator $\mathcal A_{1,\dots, m}$ in \eqref{eqanti} was not necessary, and that the positivity conjectures could be extended to a much larger world.
In this world, which we call the world of $m$-symmetric functions, the elements are symmetric in the variables $x_{m+1},x_{m+2},\dots$ while they have no special symmetry in the variables $x_1,\dots,x_m$  (the case $m=0$ corresponding to the usual symmetric functions).  To be more precise, we have suspected that, defining the $m$-symmetric Macdonald polynomials (up to a constant) as
\begin{equation} \label{eqpos}
  J_\Lambda(x;q,t) \propto \mathcal S^t_{m+1, N} E_\eta(x;q,t)
\end{equation}
there existed a natural basis of $m$-symmetric Schur functions $s_\Omega(x;t)$ (now depending on $t$) such that
\begin{equation} \label{Macdopos}
  \tilde J_\Lambda(x;q,t) = \sum_\Omega K_{\Omega \Lambda}(q,t) \, s_\Omega(x;t)
 \qquad {\rm with~} K_{\Omega \Lambda}(q,t) \in \mathbb N[q,t]
\end{equation}
with $ \tilde J_\Lambda(x;q,t)$ a plethystically modified version of $J_\Lambda(x;q,t)$. Our main problem  has been to obtain a workable definition of the $m$-symmetric Schur functions $s_\Omega(x;t)$, that is, a definition allowing to demonstrate 
their elementary properties as well as providing a precise statement for the desired positivity conjectures for the non-symmetric Macdonald polynomials\footnote{Unfortunately, the construction of the $m$-symmetric Schur functions relying extensively on tableau combinatorics  in the first version of this article  (arXiv:2206.05177v1)  was not the correct one as it yielded  counter-examples to the  positivity conjecture for the $m$-symmetric Macdonald polynomials (the smallest such instance occurring at degree 7 for $m=3$).}.
 Before  describing the main ingredients of our solution to this problem,  we should note that the $m$-symmetric Macdonald polynomials have also been considered in \cite{BG,BG2}, where they are called {\it partially-symmetric Macdonald polynomials}. Remarkably, it has been shown recently \cite{OW} that the partially-symmetric/$m$-symmetric Macdonald polynomials are in correspondence with certain $(\mathbb C^*)^2$-fixed point classes of the parabolic flag Hilbert schemes \cite{GM}, a result that generalizes the correspondence between Macdonald polynomials and  $(\mathbb C^*)^2$-fixed points of the Hilbert schemes \cite{Haiman}. 

Bases of the ring $R_m$ of $m$-symmetric functions are indexed by $m$-partitions
$\Lambda=(\pmb a;\lambda)$ where $\pmb a =(a_1,\dots,a_m) \in \mathbb Z_{\geq 0}^m$ is a (weak) composition and where $\lambda$ is a partition.
A natural basis of $R_m$ is provided by
$$
k_\Lambda(x;t)= H_{\pmb a}(x_1,\dots,x_m;t) s_\lambda(x)
$$
where   $H_{\pmb a}(x_1,\dots,x_m;t)$ is a non-symmetric Hall-Littlewood polynomial and where $s_\lambda(x)$ is a Schur function.
Let $\langle \cdot, \cdot \rangle_m$ be the unique
bilinear scalar product on $R_m$ such that
$$
\langle k_\Lambda(x;t), k_\Omega(x;t) \rangle_m = \delta_{\Lambda \Omega}\, t^{{\rm Inv}(\pmb a)} 
$$
where ${\rm Inv}(\pmb a)$ is the number of inversions in $\pmb a$. 
The reason to introduce this scalar product is that it proves more 
convenient to first define the dual $m$-symmetric Schur functions 
$s_\Lambda^*(x;t)$. In the dominant case ($\Lambda=(\pmb a;\lambda)$ with  $a_1 \geq a_2 \geq \cdots \geq a_{m}$),  $s_\Lambda^*(x;t)$ is simply a multi-Schur function (see Definition~\ref{defdualS}), while in the non-dominant case they can be constructed recursively using the Hecke algebra generators:
$$
T_i s^*_\Lambda(x;t) =  s^*_{\tilde \Lambda}(x;t)   \quad {\rm if~}a_i > a_{i+1}
$$
where $\tilde \Lambda=(\tilde {\pmb a};\lambda)$ with $\tilde{\pmb a}$ obtained from $\pmb a$ by interchanging $a_i$ and $a_{i+1}$.
 Their dual basis (up to a power of $t$) with respect to the scalar product $\langle \cdot, \cdot \rangle_m$, the sought-after  $m$-symmetric Schur functions $s_{\Lambda}(x;t)$ appearing in \eqref{Macdopos}, are then the 
the unique basis of $R_m$ such that
$$
\langle s_\Lambda(x;t), s^*_\Omega(x;t) \rangle_m = \delta_{\Lambda \Omega}\, t^{{\rm Inv}(\pmb a)} 
$$
This definition, even though not the most direct, will prove quite fruitful.
We have for instance that the Hecke algebra generator $T_i$ is again such that
\begin{equation} \label{relHecke}
T_i s_\Lambda(x;t) =  s_{\tilde \Lambda}(x;t)   \quad {\rm if~}a_i > a_{i+1}
\end{equation}
As such, 
it is possible to construct all $m$-symmetric Schur functions from those indexed by dominant $m$-partitions (unfortunately, we have not been able to obtain an 
explicit characterization of $s_\Lambda(x;t)$ for a dominant $\Lambda$. This was the main motivation behind the introduction of the dual $m$-symmetric Schur functions which, as we have seen, have a simple characterization as multi-Schur functions when $\Lambda$ is dominant).
We can also show that non-symmetric Hall-Littlewood polynomials are a sub-family of the  $m$-symmetric Schur functions:
$$
s_{(\pmb a;\emptyset)}(x;t) = H_{\pmb a}(x_1,\dots,x_m;t)
$$

Since $R_0 \subseteq R_1 \subseteq R_2 \subseteq \cdots$, it proves natural to consider the inclusion $i:R_m \to R_{m+1}$ and the restriction $r:R_{m+1} \to R_m$ (which essentially consists in setting $x_{m+1}=0$).
We obtain elegant formulas for the inclusion and restriction of 
the (dual) $m$-symmetric Schur functions, as well as  for the inclusion and restriction of the $m$-symmetric Macdonald polynomials (these are probably the most technical results contained in this article).  These formulas allow, among other things, to derive simple properties of the coefficients $K_{\Omega \Lambda}(q,t)$ in
\eqref{Macdopos} (see Theorem~\ref{propokostka} and Proposition~\ref{propoKostka1}), which imply in particular that $K_{\Omega \Lambda }(q,t)=K_{\mu \lambda}(q,t)$ whenever
$\Omega=(0^m;\mu)$ and $\Lambda=(0^m;\lambda)$, that is, that the usual $(q,t)$-Kostka coefficients are special cases of the coefficients $K_{\Omega \Lambda}(q,t)$.

Every non-symmetric Macdonald polynomial turns out to be an $m$-symmetric Macdonald polynomial when $m$ is large enough (when $m$ is as large as the length of the indexing composition to be more precise). The positivity conjecture expressed in \eqref{Macdopos} thus contains as a special case a generalization of the original Macdonald positivity conjecture for non-symmetric Macdonald polynomials.  But a more remarkable phenomenon is at play. For any compositions $\eta$ and $\omega$, we can define $K_{\omega \eta} (q,t)$ without any reference to $m$. We can also show that for a given $\eta$, there is only a finite number of distinct coefficients $K_{\omega \eta} (q,t)$,  and that, as long as $m$ is large enough, there is no loss of information in working modulo
the linear span $\mathcal L_m$ of $m$-symmetric symmetric Schur functions $s_{\Lambda}(x;t)$ such that $\lambda \neq \emptyset$
in $\Lambda=(\pmb a;\lambda)$. The positivity conjecture \eqref{Macdopos} thus implies the following positivity conjecture for non-symmetric Macdonald polynomials (properly normalized and plethystically modified) in terms of non-symmetric Hall-Littlewood polynomials
$$
\tilde J_\eta(x;q,t) = \sum_{\omega \in \mathbb Z_{\geq 0}^m} K_{\omega \eta}(q,t) H_{\omega}(x_1,\dots,x_m;t)  \mod \mathcal L_m, \qquad {\rm with~} K_{\omega \eta}(q,t) \in \mathbb N[q,t]
$$
since, as we mentioned previously, $s_{(\omega;\emptyset)}(x;t)=H_{\omega}(x_1,\dots,x_m;t)$.

When $q=t=1$, the usual $(q,t)$-Kostka coefficient $K_{\mu \lambda}(q,t)$ is equal to the number of standard tableaux of shape $\mu$ (in representation-theoretic terms, this corresponds to 
the dimension of the irreducible representation of the symmetric group indexed by the partition $\mu$).  We will show
in a forthcoming article \cite{CL} that this property extends to the $m$-symmetric world.  To be more precise, we will show that
when $q=t=1$, $K_{\Omega \Lambda}(q,t)$ is equal to the number of standard tableaux of shape $\pmb b \cup \mu$, where $\pmb b \cup \mu$ is the partition obtained by reordering the entries of the concatenation of $\pmb b$ and $\mu$, where 
$\Omega=(\pmb b;\mu)$ (in the case of compositions, $K_{\omega \eta}(1,1)$ is the number of standard tableaux of shape $\omega^+$, where $\omega^+$ is the partition obtained by reordering the entries of $\omega$).  This suggests that
the positivity conjectures in the $m$-symmetric case
could still
be  connected to the representation theory of the symmetric group.

Our goal in introducing the positivity conjecture \eqref{Macdopos} was to
define a much larger framework in which the extra structure could potentially lead to a combinatorial interpretation for the $(q,t)$-Kostka coefficients.
This extra structure seems for instance to include special recursions and Butler-type rules for the coefficients $K_{\Omega \Lambda}(q,t)$ (see Conjectures~\ref{conjecf1}, \ref{conjecf2} and \ref{conjecf3}). We are thus still hopeful that this much richer framework holds the key to unraveling the mysterious combinatorics of the $(q,t)$-Kostka coefficients.

The article, being already quite long,
does not contain the $t=1$ case (to be treated in \cite{CL}), or
the deeper properties of the $m$-symmetric Macdonald polynomials (such as the orthogonality with respect to a natural scalar product, the formulas for the squared-norm and the evaluation, etc.), which are studied in \cite{CL2}.
We have instead focused on establishing properties of the $K_{\Omega \Lambda}(q,t)$ coefficients, and, as such, have restricted ourselves to only presenting results that were necessary to establish such properties.

The outline of the article is the following.  In Section~\ref{secnonsym}, we
give the necessary background on double affine Hecke algebras and 
non-symmetric Macdonald polynomials. It is interesting to note that we use a non-standard diagrammatic representation of compositions that, apart from providing a more elegant expression for the eigenvalues of the Cherednik operators, 
will prove useful when introducing the integral form of the non-symmetric Macdonald polynomials. In Section~\ref{sect2}, we introduce the ring of $m$-symmetric functions. Generalizations of elementary concepts in symmetric function theory such as that of partition, Ferrers' diagram, and dominance order are presented.
Various extensions of simple bases of the ring of symmetric functions, such as that of monomials and power-sums, are also given.  The $m$-symmetric Macdonald polynomials are defined in Section~\ref{secmsymMacod}. This somewhat technical section culminates in a proof (using the action of certain eigenoperators involving Cherednik operators on monomials) that the $m$-symmetric Macdonald polynomials  are unitriangularly related to the $m$-symmetric monomial  basis, and thus form a basis of the ring of $m$-symmetric functions.  The (dual) $m$-symmetric Schur functions can finally be presented in Section~\ref{secdualschur}.
The dual  $m$-symmetric Schur functions are first introduced using multi-Schur functions and the Hecke algebra generators. After showing that the dual  $m$-symmetric Schur functions form a basis of the ring of $m$-symmetric functions, the $m$-symmetric Schur function can be defined as their dual with respect to a certain scalar product. 
Properties of the (dual) $m$-symmetric Schur functions are then derived in Section~\ref{secProperties} such as the action of the Hecke algebra generators on the (dual) $m$-symmetric Schur functions.  Simple formulas for the  inclusion and restriction of the  (dual) $m$-symmetric Schur functions are also established. Our main conjecture, a positivity conjecture for $m$-symmetric Macdonald polynomials is featured in Section~\ref{secMain} after a notion of plethysm and the right normalization have been defined.  Various relations on the Kostka coefficients $K_{\Omega \Lambda}(q,t)$ are later given in Section~\ref{secKostkaprop}. These are the hardest results contained in this article (apart from actually finding the right definition of the $m$-symmetric Schur functions), as  
some of the relations depend on very technical results on $m$-symmetric Macdonald polynomials that are relegated to Appendix~\ref{secAppendix} in order not to interrupt the flow of the presentation.  The case when $m$ is large, in which case the $m$-symmetric Macdonald polynomials reduce to non-symmetric Macdonald polynomials,  is studied in Section~\ref{seclargem}.  It is shown that the positivity conjecture on $m$-symmetric Macdonald polynomials, modulo a subspace spanned by certain $m$-symmetric Schur functions, now expresses how a non-symmetric Macdonald polynomial expands into  non-symmetric Hall-Littlewood polynomials. Finally, Section~\ref{secButler} contains conjectures on the coefficients $K_{\Omega \Lambda}(q,t)$, including two rules reminiscent of  Butler's rule on the $(q,t)$-Kostka coefficients.

\section{Non-symmetric Macdonald polynomials}  \label{secnonsym}

The non-symmetric Macdonald polynomials can be defined as the common eigenfunctions of the Cherednik operators \cite{Che}, which are operators that belong to the double affine Hecke algebra and act on the ring $\mathbb Q(q,t)[x_1,\dots,x_N]$.  We now give the relevant definitions \cite{Mac2,Mar}.  
Let the exchange operator $K_{i,j}$ be such that
$$K_{i,j} f(\dots, x_i,\dots,x_j,\dots)= f(\dots, x_j,\dots,x_i,\dots)$$
We then define the generators $T_i$ of the affine Hecke algebra as
\begin{equation} \label{eqTi}
T_i=t+\frac{tx_i-x_{i+1}}{x_i-x_{i+1}}(K_{i,i+1}-1),\quad i=1,\ldots,N-1,
\end{equation}
and
$$
 {T_0=t+\frac{qtx_N-x_1}{qx_N-x_1}(K_{1,N}\tau_1\tau_N^{-1}-1)}\, ,
$$
where $\tau_i f(x_1,\dots, x_i,\dots,x_N)= f(x_1,\dots, qx_i,\dots,x_N)$ is the $q$-shift operator.
The $T_i$'s satisfy the relations  {($0\leq i\leq N-1$)}:
\begin{align*} &(T_i-t)(T_i+1)=0\nonumber\\
&T_iT_{i+1}T_i=T_{i+1}T_iT_{i+1}\nonumber\\
&T_iT_j=T_jT_i \, ,\quad i-j \neq \pm 1 \mod N
\end{align*}
where the indices are taken modulo $N$.
To define the Cherednik operators, we also need to introduce
the operator $\omega$ defined as:  
$$
\omega=K_{N-1,N}\cdots K_{1,2} \, \tau_1.
$$
We note that $\omega T_i=T_{i-1}\omega$ for $i=2,\dots,N-1$.

We are now in position to define the Cherednik operators:
$$
Y_i=t^{-N+i}T_i\cdots T_{N-1}\omega T_1^{-1}\cdots T_{i-1}^{-1},
$$
where 
$$
 T_j^{-1}=t^{-1}-1+t^{-1}T_j,
$$
 which follows from the quadratic relation satisfied by the generators
 of the  Hecke algebra.
 The Cherednik operators obey the following  relations:
 \begin{align} \label{tsym1}
T_i \, Y_i&= Y_{i+1}T_i+(t-1)Y_i\nonumber \\
T_i \, Y_{i+1}&= Y_{i}T_i-(t-1)Y_i\nonumber \\
T_i Y_j & = Y_j T_i \quad {\rm if~} j\neq i,i+1.
\end{align}
It can be easily deduced from these relations that
\begin{equation}\label{TYi}
(Y_i+Y_{i+1})T_i= T_i (Y_i+Y_{i+1}) \qquad {\rm and } \qquad (Y_i Y_{i+1}) T_i =
T_i (Y_i Y_{i+1}). 
\end{equation}

An element  $\eta=(\eta_1,\dots,\eta_N)$ of $\mathbb Z_{\geq 0}^{N}$ is called a (weak) composition with $N$ 
parts (or entries).
It will prove convenient to represent a composition by a Young (or Ferrers) diagram.  Given a composition $\eta$ with $N$ parts, let $\eta^+$ be the partition obtained by reordering the entries of $\eta$.  The diagram corresponding to $\eta$ is the Young diagram of $\eta^+$ with an $i$-circle (a circle filled with an $i$) added to the right of the row of size $\eta_i$ (if there are many rows of size $\eta_i$, the circles are ordered from top to bottom in increasing order).  For instance, given $\eta=(0,2,1,3,2,0,2,0,0)$, we have
 $$
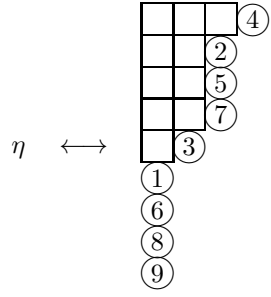
\eta \quad \longleftrightarrow  \quad {\tableau[scY]{&& & \bl \cercle{4}\\& & \bl \cercle{2} \\& & \bl \cercle{5}\\ & & \bl \cercle{7} \\ & \bl \cercle{3} \\ \bl \cercle{1} \\ \bl \cercle{6} \\ \bl \cercle{8} \\ \bl \cercle{9}}}
 $$

The Cherednik operators $Y_i$'s commute with each other, $[Y_i,Y_j]=0$,
and can be simultaneously diagonalized. Their eigenfunctions are the
 (monic) non-symmetric Macdonald polynomials (labeled by compositions).
For $x=(x_1,\dots,x_N)$, 
the non-symmetric Macdonald polynomial $E_\eta(x;q,t)$ is 
the 
unique polynomial with rational coefficients in $q$ and $t$ 
that is triangularly related to the monomials
\begin{equation} \label{orderE}
E_\eta(x_1,\dots,x_N;q,t)=x^\eta+\sum_{\nu\prec\eta}b_{\eta\nu}(q,t) \, x^\nu
\end{equation}
and that satisfies, for all $i=1,\dots,N$, 
\begin{equation} 
  Y_i E_\eta=\bar \eta_iE_\eta,\qquad\text{where}\qquad  \bar\eta_i =q^{\eta_i}t^{1-r_\eta(i)} \label{eigenvalY}
\end{equation}
  with $r_\eta(i)$ standing for the row (starting from the top) in which the $i$-circle appears in
  the diagram of $\eta$.
 The order on compositions is defined as follows:
 $$
   \nu\prec\eta\quad \text{ {iff}}\quad \nu^+<\eta^+\quad \text{or} \quad \nu^+=\eta^+\quad \text{and}\quad w_\eta < w_\nu,
 $$
 where $w_{\eta}$ is the unique permutation of minimal length such 
 that $\eta = w_{\eta} \eta^+$ ($w_{\eta}$ permutes the entries of $\eta^+$), and where the order on permutations is the Bruhat order on the symmetric group ($w_\eta {<} w_\nu$ iff
$w_{\eta}$ has a reduced decomposition which is a  {proper} subword  of a reduced decomposition of $w_{\nu}$).  The Cherednik operators have a triangular action on monomials \cite{Mac1}, that is,
\begin{equation} \label{triangY}
Y_i x^\eta =   \bar \eta_i x^\eta + {\rm ~smaller~terms}
\end{equation}
where ``smaller terms'' means that the remaining monomials $x^\nu$  appearing in the expansion are such that $\nu \prec \eta$.
  
The following three properties of the non-symmetric Macdonald polynomials will be needed below.
The first one expresses the stability of the polynomials $E_\eta$ with respect to the number of variables  (see e.g. \cite[eq. (3.2)]{Mar}):
\begin{equation} \label{property1}
E_\eta (x_1,\dots,x_{N-1},0;q,t) =
\left \{ 
\begin{array}{ll}
E_{\eta_-} (x_1,\dots,x_{N-1};q,t)
& {\rm if~} 
\eta_N = 0\, , \\
0 & {\rm if~} \eta_N \neq 0\, .
\end{array} \right.
 \end{equation}
where $\eta_-=(\eta_1,\ldots, \eta_{N-1}) $.   
The second one gives the action of the operators $T_i$ on $E_\eta$.   It is a formula that will be fundamental for our purposes \cite{BF}: 
\begin{equation} \label{property2}
T_i E_{\eta} = \left\{ 
\begin{array}{ll}
\left(\frac{t-1}{1-\delta_{i,\eta}^{-1}}\right) E_\eta + t E_{s_i \eta} & {\rm if~} 
\eta_i < \eta_{i+1} \, ,  \\
t E_{\eta} &  {\rm if~} 
\eta_i = \eta_{i+1} \, ,\\
\left(\frac{t-1}{1-\delta_{i,\eta}^{-1}}\right) E_\eta + \frac{(1-t{\delta_{i,\eta}})(1-t^{-1}\delta_{i,\eta})}{(1-{\delta_{i,\eta}})^2} E_{s_i \eta} & {\rm if~} 
\eta_i > \eta_{i+1} \, ,
\end{array} \right. 
\end{equation}
$\delta_{i,\eta}=\bar \eta_i/\bar \eta_{i+1}$ {and} $s_i \eta=(\eta_1,\dots,\eta_{i-1},\eta_{i+1},\eta_i,\eta_{i+2},\dots,\eta_N)$.
The third property, together with the previous one,  allows one to construct the non-symmetric Macdonald polynomials recursively.  Given 
$\Phi_q=t^{1-N}T_{N-1}\cdots T_1 x_1$, we have that \cite{BF}
\begin{equation} \label{eqPhi}
 \Phi_q E_\eta(x;q,t) = t^{r_\eta(1)-N}  E_{\Phi \eta}(x;q,t)
\end{equation}
where $\Phi \eta=(\eta_2,\eta_3,\dots,\eta_{N-1},\eta_1+1)$.

Finally, we introduce the $t$-symmetrization operator:
\begin{equation} \label{symop}
\mathcal S^t_{m+1,N}=\sum_{\sigma\in S_{N-m}}T_\sigma
\end{equation}
where the sum is over the permutations in the symmetric group $S_{N-m}$
and
where
$T_\sigma=T_{i_1+m}\cdots T_{i_\ell+m}$  if $\sigma=s_{i_1}\cdots s_{i_\ell}$ is a reduced expression.
    We stress that there is a shift by $m$ in the indices of the Hecke algebra generators. Essentially, the $t$-symmetrization operator $\mathcal S^t_{m+1,N}$ acts on the variables $x_{m+1},x_{m+2},\dots, x_N$ while the usual $t$-symmetrization operator $\mathcal S^t_{1,N}$ acts on all variables $x_{1},x_{2},\dots, x_N$.
    \begin{remark}  \label{remarksym}
      We have by  \eqref{eqTi} that if any polynomial $f(x_1,\dots,x_N)$ is such that  $T_i f(x_1,\dots,x_N)=t  f(x_1,\dots,x_N)$, then $f(x_1,\dots,x_N)$ is symmetric in the variables $x_i$ and $x_{i+1}$.  As such, from \eqref{property2}, $E_\eta(x_1,\dots,x_N)$ is symmetric in the variables $x_i$ and $x_{i+1}$ whenever $\eta_i=\eta_{i+1}$.  We also have that $ \mathcal S^t_{m+1,N} f(x_1,\dots,x_N)$ is symmetric in the variables  $x_{m+1},x_{m+2},\dots, x_N$ for any polynomial $f(x_1,\dots,x_N)$ given that it can easily be checked that
\begin{equation}   \label{TionSymt}
T_i \, \mathcal S^t_{m+1,N} = t \, \mathcal S^t_{m+1,N} \qquad {\rm for~}i=m+1,\dots,N-1
\end{equation}
\end{remark}
The $t$-symmetrization operator satisfies the following well-known useful relations (which can be deduced from the theory of minimal coset representatives)
\begin{equation} \label{extraN}
  (1+T_N+ T_{N-1}T_N + \cdots + T_{m+1} \cdots T_{N-1}T_N )  \, \mathcal S^t_{m+1,N} =  \mathcal S^t_{m+1,N+1}
\end{equation}
and
\begin{equation} \label{extraN2}
  (1+T_m+ T_{m+1}T_m + \cdots + T_{N-1} \cdots T_{m+1} T_m )  \, \mathcal S^t_{m+1,N} =  \mathcal S^t_{m,N}
\end{equation}
 We should note that, up to a constant, the usual Macdonald polynomial, $P_\lambda(x_1,\dots,x_N;q,t)$, is obtained by $t$-symmetrizing a non-symmetric Macdonald polynomial:
$$
P_\lambda(x;q,t) \propto \mathcal S_{1,N}^{t} \, E_\eta(x;q,t)
$$
where $\eta$ is any composition that rearranges to $\lambda$.

\section{The ring of $m$-symmetric functions} \label{sect2}

Let $\mathbf \Lambda=\mathbb Q(q,t)[h_1,h_2,h_3,\dots]$ be the ring of symmetric functions in the variables $x_1,x_2,x_3,\dots$ (the standard references on symmetric functions are \cite{M,Stan}), where
$$
h_r=h_r(x_1,x_2,x_3,\dots) = \sum_{i_1 \leq i_2 \leq \cdots \leq i_r} x_{i_1} x_{i_2} \cdots x_{i_r}
$$
 Bases of $\mathbf \Lambda$ are 
indexed by partitions $\lambda=(\lambda_1\geq\dots\geq\lambda_k>0)$ 
whose degree $\lambda$ is $|\lambda|=\lambda_1 +\cdots +\lambda_k$
and whose length $\ell(\lambda)=k$.  Each partition $\lambda$ has an 
associated Young diagram with $\lambda_i$ lattice squares in the $i^{th}$ 
row, from top to  bottom (English notation).  Any lattice square $(i,j)$ 
in the $i$th row and $j$th column of a Young diagram is called a cell.
The conjugate of $\lambda$, denoted $\lambda'$, is the reflection of
$\lambda$ about the main diagonal.  
 The partition
$\lambda\cup\mu$ is the non-decreasing rearrangement of the parts 
 of $\lambda$ and $\mu$.  If the partition $\mu$ is contained in the partition $\lambda$, then the skew diagram $\lambda/\mu$ is the diagram obtained by removing the diagram corresponding to $\mu$ from the diagram of $\lambda$.
 The dominance order $\geq$ is defined on partitions by
$\lambda\geq \mu$ when $\lambda_1+\cdots+\lambda_i\geq
\mu_1+\cdots+\mu_i$ for all $i$, and $|\lambda|=|\mu|$.

The ring $\mathbf \Lambda$ can be thought as a subring of the ring of formal power series $\mathbb Q(q,t)[[x_1,x_2,x_3,\dots]]$ since it consists of the
elements of  $\mathbb Q(q,t)[[x_1,x_2,x_3,\dots]]$ that are symmetric in the variables $x_{1},x_{2},x_{3},\dots$ and have bounded degree. In this spirit, we will define
the ring $R_m$ of $m$-symmetric functions as the subring of $\mathbb Q(q,t)[[x_1,x_2,x_3,\dots]]$ 
made  of formal power series that are symmetric  in the variables $x_{m+1},x_{m+2},x_{m+3},\dots$ and have bounded degree.
In other words, we have
$$R_m \simeq \mathbb Q(q,t)[x_1,\dots,x_m] \otimes \mathbf \Lambda_m$$
where
$\mathbf \Lambda_m$ is the ring of symmetric functions in the variables $x_{m+1},x_{m+2},x_{m+3},\dots$.  The ring $R_m$ is graded with respect to the total degree in $x$
$$
R_m = R_m^0 \oplus  R_m^1 \oplus  R_m^2 \oplus \cdots 
$$
where $R_m^d$ is the subspace of $R_m$ made of formal power series of homogeneous degree $d$. It is  immediate that $R_0=\mathbf \Lambda$ is the usual ring of symmetric functions and that $R_0 \subseteq R_1 \subseteq R_2 \subseteq \cdots $.  Bases of $R_m$ are naturally indexed by $m$-partitions which are pairs $\Lambda=(\pmb a;\lambda)$, where 
$\pmb a= (a_1,\dots,a_m) \in \mathbb Z_{\geq 0}^m$ is a composition with $m$ parts, and where
$\lambda$ is a partition.  We will call the entries of $\pmb a$ and $\lambda$ the
non-symmetric and symmetric entries of $\Lambda$ respectively.
In the following, unless stated otherwise, $\Lambda$ and $\Omega$ will always stand respectively for the $m$-partitions $\Lambda=(\pmb a;\lambda)$ and $\Omega=(\pmb b;\mu)$. Observe that we use a different notation for  the composition $\pmb a$ with $m$ parts (which corresponds to the non-symmetric entries of 
of $\Lambda$)
than for the composition $\eta$ with $N$ parts (which will typically index a non-symmetric Macdonald polynomial).

Given a composition $\pmb a$ and a partition $\lambda$, $\pmb a \cup \lambda$ will denote the partition obtained by reordering the entries of the concatenation of $\pmb a$ and $\lambda$.  The degree of an $m$-partition $\Lambda$, denoted $|\Lambda|$, is the sum of the degrees of $\pmb a$ and $\lambda$, that is, 
$|\Lambda|=a_1+\dots +a_m+\lambda_1+\lambda_2+\cdots$. We also define the 
length of $\Lambda$ as $\ell(\Lambda)=m+\ell(\lambda)$. We will say that $\pmb a$ is dominant if $a_1 \geq a_2 \geq \cdots \geq a_m$, and by extension, we will say that $\Lambda=(\pmb a; \lambda)$ is dominant if $\pmb a$ is dominant.  If $\pmb a$ is not dominant, we let $\pmb a^+$ be the dominant composition obtained by reordering the entries of $\pmb a$.  We also let $\Lambda^+=(\pmb a^+;\lambda)$.

There is a natural way to represent an $m$-partition by a Young diagram. 
The diagram corresponding to $\Lambda$ is the Young diagram of $\pmb a \cup \lambda$ with an $i$-circle  added to the right of the row of size $a_i$ for $i=1,\dots,m$ (if there are many rows of size $a_i$, the circles are ordered from top to bottom in increasing order).  For instance, given $\Lambda=(2,0,2,1; 3,2 )$, we have
 $$
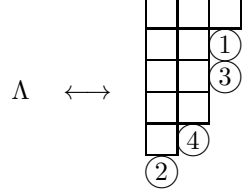
\Lambda \quad \longleftrightarrow  \quad {\tableau[scY]{&& & \bl \\& & \bl \cercle{1} \\& & \bl \cercle{3}\\ & & \bl  \\ & \bl \cercle{4} \\ \bl \cercle{2} }}
$$
Observe that when $m=0$, the diagram associated to  $\Lambda=( ; \lambda)$ coincides with the Young diagram associated to $\lambda$.
Also note that if $\eta$ is a composition with $m$ parts, then the diagram of $\eta$ coincides with the diagram of the $m$-partition $\Lambda=(\pmb a;\emptyset)$, where $\pmb a=\eta$.  We let $\Lambda^{(0)}=\pmb a \cup \lambda$, that is, $\Lambda^{(0)}$ is the partition 
obtained from the diagram of $\Lambda$ by discarding all the circles.
More generally, for $i=1,\dots,m$, we let $\Lambda^{(i)}=(\pmb a+1^i) \cup \lambda$, where $\pmb a +1^i=(a_1+1,\dots,a_i+1,a_{i+1},\dots,a_m)$. In other words, 
$\Lambda^{(i)}$ is the partition obtained from the diagram associated to $\Lambda$ by changing all of the $j$-circles, for $1 \leq j \leq i$, into squares and discarding the remaining circles.  Taking as above
$\Lambda=(2,0,2,1; 3,2)$, we have
$\Lambda^{(0)}=(3,2,2,2,1)$, $\Lambda^{(1)}=(3,3,2,2,1)$, $\Lambda^{(2)}=(3,3,2,2,1,1)$, $\Lambda^{(3)}=(3,3,3,2,1,1)$ and $\Lambda^{(4)}=(3,3,3,2,2,1)$. 
We then define the
dominance ordering on $m$-partitions to be  such that
\begin{equation} \label{deforder}
\Lambda \geq \Omega \iff \Lambda^{(i)} \geq \Omega^{(i)} \qquad {\rm for~all~}i=0,\dots,m
\end{equation}
where the order on the r.h.s.
is the usual dominance order on partitions.
Note that the meaning of dominance order will become apparent later (see Propositions~\ref{propo4} and \ref{propounitriang}). 

We let the $m$-symmetric monomial function $m_\Lambda(x)$ be defined as
$$
m_\Lambda(x) := x_1^{a_1} \cdots x_m^{a_m} \, m_\lambda (x_{m+1},x_{m+2},\dots)= x^{\pmb a}  \, m_\lambda (x_{m+1},x_{m+2},\dots)
$$
where $m_\lambda (x_{m+1},x_{m+2},\dots)$ is the usual monomial symmetric function
 in the variables $x_{m+1},x_{m+2},\dots$
$$
m_\lambda (x_{m+1},x_{m+2},\dots)= \sum_\alpha x_{m+1}^{\alpha_1} x_{m+2}^{\alpha_2} \cdots   
$$
with the sum being over all  derangements $\alpha$ of $(\lambda_1,\lambda_2,\dots,\lambda_{\ell(\lambda)},0,0,\dots)$. We emphasize that unless stated otherwise, $x$ will always stand for the variables $(x_1,x_2,\dots)$ so that
$m_\Lambda(x)=m_\Lambda(x_1,x_2,\dots)$.
It is immediate that $\{ m_\Lambda(x) \}_\Lambda$ is a basis of $R_m$, and more specifically, that
$$
\{ m_\Lambda(x) \, | \, d=|\Lambda|\}
$$
is a basis of $R_m^d$.
The following stronger statement holds.
\begin{proposition} \label{propom}
  For $\ell=1,\dots,m+1$, we have that
$
\{  x^{\pmb a}\,  m_\lambda (x_{\ell},x_{\ell+1},\dots)\}_{\pmb a,\lambda}
$
is a basis of $R_m$.  In particular, for $x=(x_1,x_2,x_3,\dots)$, we have that
$
\{  x^{\pmb a} \, m_\lambda (x)  \}_{\pmb a,\lambda}
$
is a basis of $R_m$. 
\end{proposition}  
\begin{proof}   It is obvious that  $  x^{\pmb a}\,  m_\lambda (x_{\ell},x_{\ell+1},\dots)$ belongs to $R_m$ since it is symmetric in $x_{m+1},x_{m+2},x_{m+3},\dots$.  We will show that
  $$
\{  x^{\pmb a}\,  m_\lambda (x_{\ell},x_{\ell+1},\dots) \, |\, d=|\pmb a|+|\lambda| \}
  $$
is a basis of $R_m^d$. We proceed by induction from the base case $\ell=m+1$ which, as we have seen, holds.
Having the right number of elements, we simply need to show that the  $ x^{\pmb a}\,  m_\lambda (x_{\ell},x_{\ell+1},\dots)$'s are linearly independent.  Suppose by contradiction that $\ell \leq m$ and that in $R_m^d$ we have
$$
\sum_{\pmb a, \lambda} c_{\pmb a, \lambda} \, x^{\pmb a} \, m_\lambda (x_{\ell},x_{\ell+1},\dots) =0
$$
Extracting the power of $x_\ell$ in $x^{\pmb a}$, we get 
$$
\sum_{\pmb a,\lambda }  c_{\pmb a, \lambda} \, x_\ell^{a_\ell} \, x^{ {\pmb b}} \,
m_\lambda (x_{\ell},x_{\ell+1},\dots) =0
$$
where ${ {\pmb b}}=(b_1,\dots,b_m)$ is such that $b_\ell=0$ and $b_{i}=a_i$ for $i \neq \ell$.
Let $k$ be the lowest value of $a_\ell$ such that there exists a non-zero coefficient $c_{\pmb a, \lambda}$.  Extracting $x_\ell^k$, we get that
$$
\sum_{\pmb a,\lambda }  c_{\pmb a, \lambda} \, x_\ell^{a_\ell-k} \, x^{ {\pmb b}} \, m_\lambda (x_{\ell},x_{\ell+1},x_{\ell+2},\dots) =0
$$
Letting $x_\ell=0$ in this equation, we obtain 
$$
 \sum_{\pmb a,\lambda \, : \, a_\ell=k}  c_{\pmb a, \lambda} \, x^{\pmb b} \, m_\lambda (0,x_{\ell+1},x_{\ell+2},\dots) =
 \sum_{\pmb a,\lambda \, : \, a_\ell=k}  c_{\pmb a, \lambda} \, x^{\pmb b} \, m_\lambda (x_{\ell+1},x_{\ell+2},\dots) =0
$$
 since $m_\lambda (0,x_{\ell+1},x_{\ell+2},\dots)= m_\lambda (x_{\ell+1},x_{\ell+2},\dots)$.  Hence
$$
x_\ell^k \sum_{\pmb a,\lambda \, : \, a_\ell=k}  c_{\pmb a, \lambda} \, x^{\pmb b} \, m_\lambda (x_{\ell+1},x_{\ell+2},\dots) =\sum_{\pmb a,\lambda \, : \, a_\ell=k}  c_{\pmb a, \lambda} \, x^{\pmb a} \, m_\lambda (x_{\ell+1},x_{\ell+2},\dots) = 0
$$
Note that, by the minimality of $k$, there is a non-zero coefficient in the sum.
This contradicts our induction hypothesis that
$\{  x^{\pmb a}\,  m_\lambda (x_{\ell+1},x_{\ell+2},\dots) \, |\, d=|\pmb a|+|\lambda| \}$ is
a basis of $R_m^d$.
\end{proof}  
The $m$-symmetric power sums are naturally defined as
$$
p_\Lambda(x) := x_1^{a_1} \dots x_m^{a_m} \, p_\lambda (x)= x^{\pmb a} \, p_\lambda (x)
$$
 It should be observed that
 the variables in $p_\lambda$, contrary to those of $m_\lambda$ in $m_\Lambda(x)$, start at $x_1$ instead of
$x_{m+1}$. In this expression,  $p_\lambda (x)$ is the usual power-sum symmetric function
$$
p_\lambda(x) = \prod_{i=1}^{\ell(\lambda)} \, p_{\lambda_i}(x)
$$
where $p_r(x)=x_1^r+x_2^r+\cdots$.  From Proposition~\ref{propom}, we get immediately that
$
\{ p_\Lambda(x)  \}_\Lambda
$
is a basis of $R_m$.

Another basis of the ring of symmetric functions is provided by the  Schur functions $s_\lambda(x)$
\begin{equation} \label{defSchur}
s_{\lambda}(x) = \det \left( h_{\lambda_i-i+j} \right)_{1\leq i,j \leq \ell(\lambda)}
\end{equation}
where $h_i=0$ if $i<0$.  Replacing $p_\lambda(x)$ by the Schur function $s_\lambda(x)$,
we define
$$
k_\Lambda(x) := x_1^{a_1} \dots x_m^{a_m} \, s_\lambda (x)= x^{\pmb a} \, s_\lambda (x)
$$
It is then again obvious due to Proposition~\ref{propom} that 
$
\{ k_\Lambda(x)  \}_\Lambda
$
is a basis of $R_m$. We stress that this basis is not the right extension of the Schur functions in the $m$-symmetric world (hence the notation  $k_\Lambda$ instead of  $s_\Lambda$).  A $t$-generalization of the $k_\Lambda(x)$ basis will nevertheless prove instrumental in the construction of the $m$-symmetric Schur functions.

\section{$m$-symmetric Macdonald polynomials} \label{secmsymMacod}
The $m$-symmetric Macdonald polynomials in $N$ variables are simply obtained by applying the $t$-symmetrization 
operator $\mathcal S_{m+1,N}^t$ introduced in \eqref{symop}
to non-symmetric Macdonald polynomials.
To be more specific, we associate to the $m$-partition $\Lambda= (\pmb a;\lambda)$ the composition $\eta_{\Lambda,N}={(a_1,\dots,a_m,0^n,\lambda_{\ell(\lambda)},...,\lambda_1)}$, where $n=N-m-\ell(\lambda) \geq 0$.
The corresponding $m$-symmetric Macdonald polynomial in $N$ variables  is then defined as
\begin{equation}
P_\Lambda(x_1,\dots,x_N;q,t)= \frac{1}{u_{\Lambda,N}(t)} \, \mathcal S^t_{m+1,N } \, E_{\eta_{\Lambda,N}}(x_1,\dots,x_N;q,t)  
\end{equation}
with the normalization constant $u_{\Lambda,N}(t)$  given by
\begin{equation} \label{normalization}
u_{\Lambda,N}(t)= [n]_{t^{-1}}!   \left( \prod_i [n_\lambda(i)]_{t^{-1}}! \right)
t^{(N-m)(N-m-1)/2} 
\end{equation}
where $n_\lambda(i)$ is the number of entries in $\lambda$ that are equal to $i$, and where
$$
[k]_q!=\frac{(1-q)(1-q^2)\cdots (1-q^k)}{(1-q)^k}
$$
We will see later that the normalization constant $u_{\Lambda,N}(t)$ is chosen such that the coefficient of $m_\Lambda$ in $P_\Lambda(x;q,t)$ is equal to 1. 

\begin{remark} \label{remarksymMac}
  We have that $T_i$ commutes with $\mathcal S^t_{m+1,N }$ for $i=1,\dots,m-1$. Hence, from Remark~\ref{remarksym}, if $\Lambda$ is such that  $a_i=a_{i+1}$ then $P_\Lambda(x_1,\dots,x_N;q,t)$ is symmetric in the variables $x_i,x_{i+1}$.
\end{remark}

The $m$-symmetric Macdonald polynomials are stable with respect to the number of variables. 
\begin{proposition} Let $N$ be the number of variables and suppose that $N\geq m+\ell(\lambda)$.
Then
  $$
  P_\Lambda(x_1,\dots,x_{N-1},0;q,t)=
  \left \{ 
  \begin{array}{ll}
    P_{\Lambda}(x_1,\dots,x_{N-1};q,t) & {\rm if~} N>m+\ell(\lambda) \\
    0 &   {\rm if~} N=m+\ell(\lambda) 
  \end{array} \right .  
$$
\end{proposition}  
\begin{proof}
  It was shown in \cite{Mar} (see formula (38) in Proposition 3, where the version of $u_{\Lambda,N}(t)$ used therein is slightly different from the one we use here)
  that
  $$
  \mathcal S_{m+1,N}^t E_{\eta_{\Lambda,N}} \Big|_{x_N=0} =
\left\{
  \begin{array}{ll}  
\frac{u_{\Lambda,N-1}(t)}{u_{\Lambda,N}(t)} \mathcal S_{m+1,N-1}^t E_{\eta_{\Lambda,N-1}} & {\rm if~} N>m+\ell(\lambda) \\
0 & {\rm otherwise}
  \end{array}
\right .  
$$  
The proposition then follows immediately.
\end{proof}
Since the $m$-symmetric monomial functions $m_\Lambda$ are also such that
  $$
  m_\Lambda(x_1,\dots,x_{N-1},0)=
  \left \{ 
  \begin{array}{ll}
    m_{\Lambda}(x_1,\dots,x_{N-1}) & {\rm if~} N>m+\ell(\lambda) \\
    0 &   {\rm if~} N=m+\ell(\lambda) 
  \end{array} \right .  
$$
  we have immediately the following corollary.
\begin{corollary} \label{coroinfinite}
 The expansion coefficients $d_{\Lambda \Omega}(q,t)$ in
 \begin{equation} \label{Macdod}
P_\Lambda(x_1,\dots,x_N;q,t) = \sum_{\Omega} d_{\Lambda \Omega}(q,t) m_\Omega(x_1,\dots,x_N)
  \end{equation}
 do not depend on the number of variables $N$.  Moreover, the number of terms in the r.h.s. of \eqref{Macdod} does not depend on $N$ as long as $N\geq m+|\Lambda|$ (the longest $m$-partition of a given degree $d$ being $(0^m;1^d)$).
\end{corollary}  
We see from Remark~\ref{remarksym} that the $m$-symmetric Macdonald polynomials are symmetric in the variables $x_{m+1},x_{m+2},\dots,x_N$.  From the previous corollary, we can thus let $N\to \infty$ and define the
$m$-symmetric Macdonald polynomials $P_\Lambda(x;q,t) \in R_m$, for $x=(x_1,x_2,x_3,\dots)$, as
$$
P_\Lambda(x;q,t) = \sum_{\Omega} d_{\Lambda \Omega}(q,t) m_\Omega(x)
$$
where the coefficients $d_{\Lambda \Omega}(q,t)$ are those of \eqref{Macdod}.  

\begin{remark}
  From the comment  at the end of Section~\ref{secnonsym}, when $m=0$  an $m$-Macdonald polynomial is a usual Macdonald polynomial. In the case $m=1$, an $m$-symmetric Macdonald polynomial corresponds to the $\theta_1$ coefficient of a Macdonald polynomial in superspace with fermionic degree equal to 1 \cite{BDLM,BDLM2}.
  \end{remark}

It is obvious from  \eqref{tsym1} that  $\mathcal S_{m+1,N}^t$ commutes with $Y_i$, for $i=1,\dots,m$.  It is then immediate from the definition of the $m$-symmetric Macdonald polynomials and \eqref{eigenvalY} that $P_\Lambda(x_1,\dots,x_N;q,t)$ is an eigenfunction of $Y_i$ with eigenvalue $(\bar \eta_{\Lambda,N})_i$.  Now, it is not hard to see that the diagram of $\Lambda$ is obtained from that of $\eta_{\Lambda,N}$  by removing the circles filled with an $m+1,m+2,\dots,N$. Hence, for $i=1,\dots,m$, we have
\begin{equation} \label{defeigeni}
  Y_i P_\Lambda(x_1,\dots,x_N;q,t) = \varepsilon_\Lambda^{(i)}(q,t)  P_\Lambda(x_1,\dots,x_N;q,t), \qquad {\rm with} \quad  \varepsilon_\Lambda^{(i)}(q,t)=  q^{a_i} t^{1-r_\Lambda(i)} 
\end{equation}
where we recall that $r_\Lambda(i)$ is the row in which the
$i$-circle  appears in the diagram associated to $\Lambda$.  Letting
$$
D =  Y_{m+1}+\dots+Y_N - \sum_{i=m+1}^N t^{1-i}
$$
it is also quite easy to check, using  \eqref{tsym1} and \eqref{TYi}, that $D$ commutes with $\mathcal S_{m+1,N}^t$.  Following the argument used previously in the case of $Y_i$, this implies that  $P_\Lambda(x_1,\dots,x_N;q,t)$ is an eigenfunction of $D$ with eigenvalue
$$
(\bar \eta_{\Lambda,N})_{m+1}+ \cdots + (\bar \eta_{\Lambda,N})_{N} -  \sum_{i=m+1}^N t^{1-i}
$$
Note that the $i$-circles, for $i=m+\ell(\lambda)+1,\dots,N$, are located in rows $m+\ell(\lambda)+1$ up to $N$ of the diagram of $\eta_{\Lambda,N}$, which are all of length zero.
From our previous observation that  the diagram of $\Lambda$ is obtained from that of   $\eta_{\Lambda,N}$ by removing the circles filled with an $m+1,m+2,\dots,N$  we then get that
\begin{equation} \label{defeigenD}
D \,  P_\Lambda(x_1,\dots,x_N;q,t)= \varepsilon_\Lambda^{D}(q,t)
P_\Lambda(x_1,\dots,x_N;q,t), \qquad {\rm with} \quad  
 \varepsilon_\Lambda^{D}(q,t)=
 {\sum_{i}}' q^{\Lambda^{(0)}_i} t^{1-i} - \sum_{i=m+1}^{m+\ell(\lambda)} t^{1-i}
\end{equation}
 where the prime indicates that the sum is only over the rows of the diagram of $\Lambda$ that do not end with a circle.  We stress that the eigenvalues
 $ \varepsilon_\Lambda^{(i)}(q,t)$ and $ \varepsilon_\Lambda^{D}(q,t)$  do not depend on the number $N$ of variables.

 Before proving that the $m$-symmetric Macdonald polynomials are triangularly related to the monomials in the dominance order \eqref{deforder}, we need the following.
 \begin{proposition} \label{propo4}
  The operators $Y_1,\dots,Y_m$ and $D$ have a triangular action on monomials (in the dominance order on $m$-partitions).  To be more precise,  for certain coefficients $b_{\Lambda \Omega}^{(i)}(q,t)$ and $b_{\Lambda \Omega}'(q,t)$ in $\mathbb Q[q,t]$, we have that
 \begin{equation} \label{triangmY1}
   Y_i m_\Lambda(x_1,\dots,x_N) =\varepsilon_\Lambda^{(i)}(q,t)\,  m_\Lambda(x_1,\dots,x_N) + \sum_{\Omega < \Lambda} b_{\Lambda \Omega}^{(i)}(q,t) \, m_\Omega(x_1,\dots,x_N)
\end{equation}
and
\begin{equation}\label{triangmY2}
  D \, m_\Lambda(x_1,\dots,x_N) = \varepsilon_\Lambda^{D}(q,t)\,    m_\Lambda(x_1,\dots,x_N) + \sum_{\Omega < \Lambda} b_{\Lambda \Omega}'(q,t) \, m_\Omega(x_1,\dots,x_N)
  \end{equation}
with $\varepsilon_\Lambda^{(i)}(q,t)$ and $\varepsilon_\Lambda^{D}(q,t)$ given in \eqref{defeigeni} and \eqref{defeigenD} respectively.
\end{proposition}  
\begin{proof}
  It is obvious that the operators $Y_i$ for $i=1,...,m$ and $D$ preserve $R_m$ since, from \eqref{tsym1} and \eqref{TYi},  they all commute with $T_j$ for $j=m+1,\dots,N-1$.  We thus only need to show the triangularity and that the dominant coefficients are as indicated.

 It is shown in \cite{BDLM2} (see the proof of Proposition~6 therein) that if $x^\gamma$ appears in $Y_i x^\eta$, then $\gamma^{(0)} \leq \eta^{(0)}$ and $\gamma^{(i)} \leq \eta^{(i)}$ for any $i$, where $\gamma^{(0)}=\gamma^+$ and $\gamma^{(i)}=(\gamma+1^i)^+$.
 We therefore get that a given ``smaller term'' $x^\gamma$ appearing in
\eqref{triangY} is such that 
$\gamma^{(0)} \leq \eta^{(0)}$ and $\gamma^{(j)} \leq \eta^{(j)}$ for all $j =1,\dots,m$. Observe that if $x^\gamma$ is a monomial that appears in $m_\Omega$, then  $\gamma^{(0)} = \Omega^{(0)}$ and $\gamma^{(j)} = \Omega^{(j)}$ for every $j =1,\dots,m$.   Hence, from our previous observation,
 $$
Y_i m_\Lambda = c_{i,\Lambda}(q,t) \,  m_\Lambda + \sum_{\Omega < \Lambda} b_{\Lambda \Omega}^{(i)}(q,t) m_\Omega  \qquad {\rm and} \qquad D \,  m_\Lambda =   c'_{\Lambda}(q,t)\,  m_\Lambda + \sum_{\Omega < \Lambda} b_{\Lambda \Omega}'(q,t) m_\Omega
$$
for certain coefficients $c_{i,\Lambda}(q,t)$, $ b_{\Lambda \Omega}^{(i)}(q,t)$,
 $c_{\Lambda'}(q,t)$, and $ b_{\Lambda \Omega}'(q,t)$.
It remains to show that $c_{i,\Lambda}(q,t)=\varepsilon_\Lambda^{(i)}(q,t)$ and
$c_{\Lambda}'(q,t)=\varepsilon_\Lambda^D(q,t)$.
Let $\eta$ be maximal in the order $\prec$ on compositions among all the monomials $x^\eta$ appearing in $m_\Lambda$. From \eqref{triangY}, we then have that the coefficients of $x^\eta$ in $Y_i m_\Lambda$ and  $D  \,  m_\Lambda$ are respectively equal to $\varepsilon_\Lambda^{(i)}(q,t)$ and $\varepsilon_\Lambda^D(q,t)$.  By symmetry, the coefficients of $m_\Lambda$ are as wanted.
  \end{proof}

\begin{proposition} \label{propounitriang}
We have that
\begin{equation} \label{unitriang}
P_\Lambda(x;q,t)= m_\Lambda + \sum_{\Omega < \Lambda} d_{\Lambda \Omega}(q,t) \, m_\Omega
\end{equation}
Hence, the $m$-symmetric Macdonald polynomials  form a basis of $R_m$.
\end{proposition}
\begin{proof} We prove the result when the number of variables $N$ is finite.  The infinite case will then follow from Corollary~\ref{coroinfinite}. 

  First note that the eigenvalues $\varepsilon_\Lambda^{(i)}(q,t)$, for $i=1,\dots,m$ and $\varepsilon_\Lambda^{D}(q,t)$ completely determine $\Lambda$ (the powers of $q$ in $\varepsilon_\Lambda^{D}(q,t)$ determine $\lambda$ while the power of $q$ in $\varepsilon_\Lambda^{(i)}(q,t)$ determines $a_i$).
  Suppose that there exists a term $m_\Omega$ with non-zero coefficient in the monomial expansion of $P_\Lambda$ such that $\Omega \not \leq \Lambda$ and let $\Gamma$ be maximal among those terms.  From Proposition~\ref{propo4}, the coefficients of  $m_\Gamma$ in 
  $Y_i P_\Lambda$ and $D P_\Lambda$ are respectively $\varepsilon_\Gamma^{(i)}$ and 
$\varepsilon_\Gamma^{D}$.  But from \eqref{defeigeni} and \eqref{defeigenD}, those coefficients are also equal to  $\varepsilon_\Lambda^{(i)}(q,t)$ and 
  $\varepsilon_\Lambda^{D}(q,t)$.  Hence  $\varepsilon_\Gamma^{(i)}= \varepsilon_\Lambda^{(i)}$ and $\varepsilon_\Gamma^{D}=\varepsilon_\Lambda^{D}$.  Since, as we have seen, those eigenvalues determine $\Lambda$, we have the contradiction that $\Gamma\leq \Gamma=\Lambda$. 

  Finally, we need to show that the coefficient of $m_\Lambda$ in $P_\Lambda$ is equal to 1.  It is shown in \cite{Mar} (see equation (5.35) therein) that, when $m=0$, the coefficient of  $x^\lambda$ in $\mathcal S_{1,N}^t E_{\eta_{\lambda,N}}$ is equal to  the coefficient of  $x^\lambda$ in $\mathcal S_{1,N}^t x^{\eta_{\lambda,N}}$, which is itself equal to  $u_{\lambda,N}(t)$. The same exact argument can be used to show that  the coefficient of  $x^\Lambda$ in $\mathcal S_{m+1,N}^t E_{\eta_{\Lambda,N}}$ is equal to  the coefficient of  $x^\Lambda$ in $\mathcal S_{m+1,N}^t x^{\eta_{\Lambda,N}}$, which is itself equal to  $u_{\Lambda,N}(t)$. The coefficient of $m_\Lambda$ in $P_\Lambda$ is then equal to 1 given the normalization of $P_\Lambda$.

\end{proof}

\section{Definition of the (dual) $m$-symmetric Schur functions} \label{secdualschur}
In the remainder of this article, we will use the notation
$s_i \pmb a=(a_1,\dots, a_{i+1},a_i, \dots, a_m)$, and more generally
$s_i \Lambda=(s_i \pmb a;\lambda)$.

Let $H_{\pmb a}(x_1,\dots,x_m;t)= E_{\pmb a}(x_1,\dots,x_N;0,t)$ be the  non-symmetric Hall-Littlewood polynomial (observe that the non-symmetric Macdonald polynomial $E_{\pmb a}(x_1,\dots,x_N;q,t)$ only depends on the variables $x_1,\dots,x_m$ when $q=0$ and the indexing composition has length at most $m$).   For $i=1,\dots,m-1$, the action of the Hecke operator $T_i$
on the non-symmetric Hall-Littlewood polynomials, which can be obtained by taking the limit $q=0$ in \eqref{property2}, is the following:
\begin{equation} \label{TiH}
T_i  H_{\pmb a} =
\left \{ \begin{array}{ll}
     H_{s_i \pmb a}  & {\rm if~} a_i > a_{i+1} \\
    (t-1)  H_{\pmb a} +  t H_{s_i \pmb a}   & {\rm if~} a_i < a_{i+1} \\
     t H_{\pmb a}  & {\rm if~} a_i = a_{i+1}
\end{array} \right .   
\end{equation}
We should note that the polynomial  $H_{\pmb a}(x_1,\dots,x_m;t)$ can be constructed recursively from this formula since $H_{\pmb a}=x^{\pmb a}$ when $\pmb a$ is dominant (it is not too difficult to show that it is indeed the case using \eqref{eqPhi} and \eqref{TiH}).  When $t=1$, the operator $T_i$ becomes $K_{i,i+1}$ and it is then immediate from this construction that $H_{\pmb a}(x_1,\dots,x_m;1)=x^{\pmb a}$. It is thus natural to define the following $t$-generalization of the
$k_\Lambda(x)$ basis:
$$
k_\Lambda(x;t)= H_{\pmb a}(x_1,\dots,x_m;t) s_\lambda(x)
$$
We stress that
$$
\{  k_\Lambda(x;t) \, | \, \pmb a \in \mathbb Z_{\geq 0}^m\}
$$
is indeed a basis of $R_m$ since we have seen that the special case $k_\Lambda(x;1)=k_\Lambda(x)$ is a basis of $R_m$ (that is, the $k_\Lambda(x;t)$'s are linearly independent and thus form a basis of $R_m$ by dimension considerations).

Since $s_\lambda(x)$ commutes with $T_i$ for all $i$, \eqref{TiH} implies immediately the following.
\begin{lemma} \label{propoTik}
 For $i=1,\dots,m-1$, the operator $T_i$ acts on 
 $k_\Lambda(x;t)$ as 
\begin{equation} \label{Tik}
T_i  k_{\Lambda} =
\left \{ \begin{array}{ll}
     k_{\tilde \Lambda}  & {\rm if~} a_i > a_{i+1} \\
    (t-1)  k_{\Lambda} +  t k_{\tilde  \Lambda}   & {\rm if~} a_i < a_{i+1} \\
     t k_{\Lambda}  & {\rm if~} a_i = a_{i+1}
\end{array} \right .    
\end{equation}
where $\tilde \Lambda=s_i \Lambda$.
\end{lemma}
A bilinear scalar product $\langle \cdot, \cdot \rangle_m$ on $R_m$ is defined by requiring that the $k_\Lambda(x;t)$ basis be such that:
\begin{equation} \label{scalprod}
\langle k_\Lambda(x;t),  k_\Omega(x;t)  \rangle_m = \delta_{\Lambda \Omega} t^{{\rm Inv} (\pmb a)}
\end{equation}
where
\begin{equation} \label{definv}
{\rm Inv} (\pmb a) = \#\{ i<j \, |\, a_i < a_j \}
\end{equation}
is the number of inversions in $\pmb a$.

The following proposition follows immediately from Lemma~\ref{propoTik}.
\begin{proposition} \label{propoTidual}
  For $i=1,\dots,m-1$, the operator $T_i$ is self-adjoint with respect to the scalar product
  $\langle \cdot , \cdot \rangle_m$, that is,
  $$
\langle T_i f, g \rangle_m= \langle  f, T_i g \rangle_m \quad  {\rm for~all~} f,g \in R_m
  $$
\end{proposition}  
\begin{proof} We prove that the relation holds on the basis $\{ k_\Lambda(x;t)\}_\Lambda$.
  Let $\Lambda=(\pmb a;\lambda)$ be such that $a_i>a_{i+1}$.
Using Lemma~\ref{propoTik}, we can easily check that, for $i=1,\dots,m-1$, we have 
  $$
\langle T_i k_\Lambda, k_{\tilde \Lambda} \rangle = t^{{\rm Inv} (s_i \pmb a)} = t^{1+{\rm Inv} (\pmb a)}  = \langle  k_\Lambda, T_i k_{\tilde \Lambda} \rangle
$$
which also implies that $\langle T_i k_{\tilde \Lambda}, k_{\Lambda} \rangle  = \langle  k_{\tilde \Lambda}, T_i k_{ \Lambda} \rangle$. This completes the proof since the only non-zero remaining scalar products to verify  are symmetric.
\end{proof}

We are now ready to introduce the dual $m$-symmetric Schur functions.  Once this is done, we will then be able to define the $m$-symmetric Schur functions by duality (see Definition~\ref{defmschur}).  As mentioned in the introduction, the  $m$-symmetric Schur functions will provide the right basis to state our positivity conjecture for the $m$-symmetric Macdonald polynomials (Conjecture~\ref{mainconjec}). The fact, by Remark~\ref{remarkstandard}, that the   $(q,t)$-Kostka coefficient $K_{\Omega \Lambda}(q,t)$ at $q=t=1$ counts the number of standard tableaux of a certain shape and, by Corollary~\ref{corospecial}, that  the usual $(q,t)$-Kostka coefficients $K_{\mu \lambda}(q,t)$ are special cases of the coefficients $K_{\Omega \Lambda}(q,t)$ will give  further evidence that the $m$-symmetric Schur functions  do indeed provide the right extension of the Schur functions to the $m$-symmetric world.

In the following, it will prove convenient to use the plethystic notation in which, for a symmetric function $f$ and $X=x_1+x_2+\cdots$, we let $f[X]=f(x)=f(x_1,x_2,\dots)$.  More generally, we have using this notation that $f[X+x_1+\cdots+x_k]=f(x_1,\dots,x_k,x_1,x_2,\dots)$.

Let $\nu$ be a partition of length $\ell$.  For a sequence of alphabets $X_1,\dots,X_\ell$, where $\ell=\ell(\nu)$, the multi-Schur function $s_\nu(X_1,\dots,X_\ell)$ is defined as \cite{MacS}
\begin{equation} \label{eqmulti}
s_\nu(X_1,\dots,X_\ell) = \det \Bigl( h_{\nu_i-i+j}[X_i] \Bigr)_{1 \leq i,j \leq \ell}
\end{equation}
Observe that when comparing with \eqref{defSchur}, we get that $s_\nu(X_1,\dots,X_\ell)=s_\nu(x)$
whenever $X=X_1=X_2=\cdots=X_\ell$.

\begin{definition} \label{defdualS}
  The dual $m$-symmetric Schur functions $ s_{\Lambda}^*(x;t)$
  are defined recursively in the following way.  If $\Lambda=(\pmb a;\lambda)$ is dominant, then
$$
  s_\Lambda^*(x;t)=s_\nu(X_1,\dots,X_\ell)
$$
where $\nu=\Lambda^{(0)}=\pmb a \cup \lambda$, and where
 $X_i$ stands for the alphabet $X+x_1+\cdots +x_k$ with $k$ the number of circles weakly above row $i$ in the diagram corresponding to $\Lambda$.
 Otherwise, if $a_i<a_{i+1}$, then
 \begin{equation} \label{recursis}
s^*_\Lambda(x;t)= T_i s_{\tilde \Lambda}^*(x;t)
 \end{equation}
where $\tilde \Lambda=s_i \Lambda$, which amounts to saying that
\begin{equation} \label{stsigma}
s^*_\Lambda(x;t)=  T_{\sigma^{-1}} s_{\Lambda^+}^*(x;t)
\end{equation}
where $\sigma$ is the shortest permutation such that $\sigma(\pmb a)=(a_{\sigma^{-1}(1)},\dots,a_{\sigma^{-1}(m)}) = \pmb a^+$. 
\end{definition}

\begin{example} The diagram associated to $(2,1;3,1)$ is 
$$
{\tableau[scY]{&& & \bl \\& & \bl \cercle{\rm 1} \\&  \bl \cercle{\rm 2}\\ & \bl  \\ }}
$$
from which we deduce that
   $$
s_{2,1;3,1}^*(x;t)= \left|
\begin{array}{cccc}
  h_3[X] & h_4[X] & h_5[X] & h_6[X ]\\
   h_1[X+x_1 ]& h_2[X+x_1] & h_3[X+x_1] & h_4[X+x_1] \\
   0 & h_0[X+x_1+x_2]  & h_1 [X+x_1+x_2]& h_2[X+x_1+x_2]\\
   0 & 0 & h_0[X+x_1+x_2] & h_1[X+x_1+x_2]
\end{array}   
\right| 
$$
\end{example}
\begin{remark}
The multi-Schur functions that we use in Definition~\ref{defdualS} 
are essentially flagged Schur functions \cite{LS,W} in infinite alphabets instead of finite alphabets. For instance, if $X$ were equal to $y_{1}+\cdots+y_j$, the dual $m$-symmetric Schur function $s_{2,1;3,1}^*(x;t)$ would correspond in the language of \cite{W} to the flagged Schur functions $s_{3,2,1,1}(b)$ with flags $b_1=0,b_2=1,b_3=2$ and $b_4=2$ (with the understanding that 
the variables $y_1,\dots,y_j$
would not be constrained by any of the flags).
\end{remark}
\begin{remark} \label{remarkm1} Using the connection with the flagged Schur functions mentioned in the previous remark, we can obtain a combinatorial interpretation \cite{CL} for  the expansion of $s_\Lambda^*(x;t)$ in terms of the $k_\Omega(x)$'s, where we recall that $k_\Omega(x)=k_\Omega(x;1)$. We describe this combinatorial interpretation in the case $m=1$ since we will need it in Section~\ref{secKostkaprop}.   In that case,  is not difficult to see that Proposition\ref{proprecursion} yields
$$
s_{(a_1;\lambda)}^*(x;t) = \sum_{(b_1;\mu)} k_{(b_1;\mu)}(x)=\sum_{(b_1;\mu)} k_{(b_1;\mu)}(x;t)
$$
where the sum is over all $(b_1;\mu)$'s such that $(\lambda \cup (a_1))/\mu$ is a horizontal $b_1$-strip whose cells all lie within the first $a_1$ columns.  Observe that the second equality holds since
 all $m$-partitions are dominant when $m=1$, which implies  that 
$k_\Omega(x;t)$ does not depend on $t$ in that case. 
\end{remark}

We first prove an elementary property of the dual $m$-symmetric Schur functions.
\begin{proposition} \label{propall0}
  If $\Lambda=(0,0,\dots,0; \lambda)$ then $s_\Lambda^*(x;t)=k_\Lambda(x;t)=s_\lambda(x)$.
\end{proposition}
\begin{proof}
 If $\ell$ is the length of $\lambda$, we have in this case from Definition~\ref{defdualS} that
$$
s_\Lambda^*(x;t)=s_\lambda(X_1,\dots,X_\ell)
$$
where $X=X_1=\cdots=X_\ell$ given that all the circles are below row $\ell$ in the diagram of $\Lambda$. It is then immediate that $s_\Lambda^*(x;t)=s_\lambda(x)$ from the observation after \eqref{eqmulti}.  
\end{proof}

The following recursions due to Luis Pena \cite{P} will prove useful. In the remainder of the article, $\lambda\setminus \lambda_j$ will stand for the partition obtained by removing the entry $\lambda_j$ from the partition $\lambda$.
\begin{proposition} \label{proprecursion}
  Let $\Lambda=(\pmb a;\lambda)$ be dominant. If $a_m=0$ then
  $$ s_\Lambda^*(x;t)=s_{\Lambda_-}^*(x;t)$$
  where  $\Lambda_-=(a_1,\dots,a_{m-1};\lambda)$.  Otherwise,
  $$  s_\Lambda^*(x;t) = s_{\Lambda_D}^*(x;t)+x_m s_{\Lambda_L}^*(x;t)$$
  where $\Lambda_D=(a_1,\dots,a_{m-1},\lambda_r; (\lambda \setminus \lambda_r) \cup (a_m))$, with $\lambda_r$ the largest entry in $\lambda$ strictly smaller than $a_m$ (note that $\lambda_r$ can be equal to 0), and where
    $\Lambda_L= (a_1,\dots,a_{m-1},a_m-1;\lambda)$.
\end{proposition}  
\begin{example} Using again $\Lambda=(2,1;3,1)$, we have that
$$
\Lambda_D \longleftrightarrow {\tableau[scY]{&& & \bl \\& & \bl \cercle{\rm 1} \\&  \bl \\ & \bl  \\  \bl \cercle{\rm 2}}} \qquad {\rm and} \qquad \Lambda_L \longleftrightarrow {\tableau[scY]{&& & \bl \\& & \bl \cercle{\rm 1} \\&  \bl  \\  \bl \cercle{\rm 2}}} 
$$
Hence
   $$
s_{2,1;3,1}^*(x;t)= s_{2,0;3,1,1}^*(x;t)+x_2 s_{2,0;3,1}^*(x;t)
$$
\end{example}

\begin{proof}  If $a_m=0$ then $\Lambda^{(0)}=\hat \Lambda^{(0)}$. The result is then immediate since, in the definition of the dual $m$-symmetric Schur functions,  the alphabets $X_i$  are the same for  $s_\Lambda^*(x;t)$ and $s_{\hat \Lambda}^*(x;t)$.

 We now consider the case  $a_m>0$. Let  $j$ be the row that ends with the circle $m$ in the diagram of $\Lambda$.  Observe that there are no row ending with a circle below row $j$ ($\Lambda$ is dominant).
    Suppose that $\nu_j=\nu_{j+1}=\cdots=\nu_{s}$ and that $\nu_s > \nu_{s+1}$, where $\nu=\Lambda^{(0)}$.
 It is known that \cite{MacS} 
 \begin{equation} \label{idenmS}
s_\nu(Y_1,\dots, Y_\ell)=s_\nu(Y_1,\dots,Y_{j-1},Y_{s},\dots,Y_{s},Y_s,Y_{s+1},\dots, Y_\ell)
\end{equation}
if $Y_k$ is a subalphabet of $Y_{s}$ which differs from $Y_s$ by no more than 
$s-k$ elements for all $ j \leq k \leq s$.
We have that
$X_i=X+x_1+\cdots+x_m$ for all $i \geq j$  and that $X_{j-1}=X+x_1+\cdots+x_{m-1}$.
The previous equation thus yields
\begin{align*}
  s_\nu(X_1,\dots, X_\ell)& =s_\nu(X_1,\dots,X_{j-1},X_{s},\dots,X_{s},X_s,X_{s+1},\dots, X_\ell)\\
  & =s_\nu(X_1,\dots,X_{j-1},X_{j-1},\dots,X_{j-1},X_s,X_{s+1},\dots, X_\ell)  
\end{align*}
We now use the simple identity
 \begin{equation*}
   h_i[X_s]= h_i[X_{j-1}+x_m]= h_i[X_{j-1}] +x_m  h_{i-1}[X_s]
 \end{equation*}
in row $s$ of $s_\nu(X_1,\dots, X_\ell)$ to deduce that
\begin{align*}
  s_\nu(X_1,\dots, X_\ell) & = s_\nu(X_1,\dots,X_{j-1},X_{j-1},\dots,X_{j-1},X_{j-1},X_{s+1},\dots, X_\ell) \\
& \qquad \qquad  +x_m  s_{\tilde \nu}(X_1,\dots,X_{j-1},X_{j-1},\dots,X_{j-1},X_{s},X_{s+1},\dots, X_\ell) 
\end{align*}
where  $\tilde \nu=(\nu_1,\dots,\nu_{s-1},\nu_{s}-1,\nu_{s+1},\cdots,\nu_\ell)$ is a partition by hypothesis. Setting $\lambda_r=\nu_{s+1}$, it is not too difficult to see that the previous equation then corresponds to
$s_\Lambda^*(x;t)  = s_{\Lambda_D}^*(x;t)+x_m s_{\Lambda_L}^*(x;t)$.
\end{proof}
We now establish the triangularity of the dual $m$-symmetric Schur functions in the $k_\Lambda(x;t)$ basis.
\begin{proposition} \label{propdualtriang}
For any $\Lambda=(\pmb a;\lambda)$, we have that
\begin{equation} \label{eqdefD}
s_\Lambda^*(x;t) = k_\Lambda(x;t) + \sum_{\Omega} D_{\Lambda \Omega}^*(t) k_\Omega(x;t)
\end{equation}
where all the $\Omega=(\pmb b;\mu)$'s appearing in the sum are such that $|\pmb b| < |\pmb a|$.
\end{proposition}
\begin{proof}
 We proceed by induction on the size of $\pmb a$.  From Proposition~\ref{propall0}, the result holds when $|\pmb a|=0$.  Now, suppose that $\Lambda$ is dominant.
 From Proposition~\ref{proprecursion}, we have that
\begin{equation} \label{proof1}
s_\Lambda^*(x;t) =s_{\Lambda_D}^*(x;t)+x_m s_{\Lambda_L}^*(x;t)\end{equation}
where $\Lambda_D=(a_1,\dots,a_{m-1},\lambda_r; (\lambda \setminus \lambda_r) \cup (a_m))$, with $\lambda_r$ the largest entry in $\lambda$ strictly smaller than $a_m$, and where
    $\Lambda_L= (a_1,\dots,a_{m-1},a_m-1;\lambda)$. By definition of $\lambda_r$, we have that
    $|(a_1,\dots,a_{m-1},\lambda_r)|< |\pmb a|$. Hence by induction, 
    \begin{equation}
    \label{proof2} s_{\Lambda_D}^*(x;t)
  = k_{\Lambda_D}(x;t) + \sum_{\Omega} D_{\Lambda_D \Omega}^*(t) k_\Omega(x;t)=\sum_{\Gamma} D_{\Lambda_D \Gamma}^*(t) k_\Gamma(x;t)
      \end{equation}
where all the $\Gamma=(\pmb b;\mu)$'s appearing in the last sum are such that $|\pmb b| \leq  |(a_1,\dots,a_{m-1},\lambda_r)|<|\pmb a|$. 

Since  $|(a_1,\dots,a_{m-1},a_m-1)|< |\pmb a|$, we also get by induction that \begin{equation*} 
    s_{\Lambda_L}^*(x;t)
  = k_{\Lambda_L}(x;t) + \sum_{\Omega} D_{\Lambda_L \Omega}^*(t) k_\Omega(x;t)
      \end{equation*}
where all the $\Omega=(\pmb b;\mu)$'s appearing in the sum are such that $|\pmb b| <  |(a_1,\dots,a_{m-1},a_m-1)|$. Since $\Lambda_L$ is dominant, we have that $k_{\Lambda_L}(x;t)=x^{\Lambda_L} s_\lambda(x)$, which implies that $x_m k_{\Lambda_L}(x;t)=x^\Lambda s_\lambda(x)=k_\Lambda(x;t)$. Hence, we get from the previous equation that
\begin{equation} \label{proof3}
x_m s_{\Lambda_L}^*(x;t)
=k_{\Lambda}(x;t) +x_m \sum_{\Omega} D_{\Lambda_L \Omega}^*(t) k_\Omega(x;t)
  = k_{\Lambda}(x;t) + \sum_{\Gamma} E_{\Lambda_L \Gamma}^*(t) k_\Gamma(x;t)
  \end{equation}
where all the $\Gamma=(\pmb c;\mu)$'s appearing in the sum are such that $|\pmb c| <  |(a_1,\dots,a_{m-1},a_m-1)|+1=|\pmb a|$. The result thus holds in the dominant case from \eqref{proof1}, \eqref{proof2} and \eqref{proof3}.

In the non-dominant case, we have from \eqref{stsigma} and Lemma~\ref{propoTik} that
$$s^*_\Lambda(x;t)=  T_{\sigma^{-1}} s_{\Lambda^+}^*(x;t)= T_{\sigma^{-1}}  \left(k_{\Lambda^+}(x;t) + \sum_{\Omega} D_{\Lambda^+ \Omega}^*(t) k_\Omega(x;t)   \right) = k_\Lambda(x;t) + \sum_{\Gamma} D_{\Lambda \Gamma}^*(t) k_\Gamma(x;t) $$
Since $\Lambda^+$ is dominant, we have that all the $\Omega=(\pmb b;\mu)$'sappearing in the first sum are such that $|\pmb b|<  |\pmb a|$. But this implies that all the $\Gamma=(\pmb c;\nu)$'s appearing in the second sum are also such that $|\pmb c|<  |\pmb a|$ given that the Hecke operators $T_i$ preserve the degree of a polynomial and commute with symmetric functions.

\end{proof}

The following important proposition is now immediate.
 \begin{proposition}The  dual $m$-symmetric Schur functions  form a basis of $R_m$.
\end{proposition}  
 \begin{proof} The unitriangularity  in Proposition~\ref{propdualtriang} immediately implies that the dual $m$-symmetric Schur functions  form a basis of $R_m$.
  \end{proof}  

Since the dual $m$-symmetric Schur functions do in fact form a basis, we can define the $m$-symmetric Schur functions, $s_\Lambda(x;t)$, as their dual basis (up to a power of $t$) with respect to the scalar product \eqref{scalprod}.
 \begin{definition} \label{defmschur} 
The  $m$-symmetric Schur functions $s_\Lambda(x;t)$ are defined to be such that
 $$
\langle s_\Lambda(x;t), s_\Omega^*(x;t) \rangle_m = \delta_{\Lambda \Omega} \, t^{{\rm Inv}(\pmb a) }
 $$
Equivalently, by duality, the $m$-symmetric Schur functions $s_\Lambda(x;t)$
are the unique basis of $R_m$ such that
\begin{equation} \label{eqDs}
k_{\Omega} (x;t)= s_{\Omega}(x;t)+\sum_{\Lambda} D_{\Lambda \Omega}(t) \, s_{\Lambda}(x;t)
\end{equation}
where the coefficients $D_{\Lambda \Omega}(t)$ are, up to a $t$-power, the 
 coefficients $D^*_{\Lambda \Omega}(t)$ appearing in \eqref{eqdefD}:
\begin{equation} \label{DDs}
D_{\Lambda \Omega}(t) = t^{{\rm Inv}(\pmb b)-{\rm Inv}(\pmb a)} D^*_{\Lambda \Omega}(t)
\end{equation}
 \end{definition}

\section{Properties of the (dual) $m$-symmetric Schur functions} \label{secProperties}

We will first see that the action of $T_i$ on $s_\Lambda^*(x,t)$ is exactly as its action on $k_\Lambda(x;t)$ given in Lemma~\ref{Tik}.
\begin{proposition} \label{propoTstar}
  For $i=1,\dots,m-1$, the operator $T_i$ is such that
$$
T_i  s_\Lambda^* =
\left \{ \begin{array}{ll}
      s_{\tilde \Lambda}^*  & {\rm if~} a_i > a_{i+1} \\
    (t-1)  s_{\Lambda}^* +   t\, s_{\tilde \Lambda}^*   & {\rm if~} a_i < a_{i+1} \\
     t s_{\Lambda}^*  & {\rm if~} a_i = a_{i+1}
\end{array} \right .    
$$
where $\tilde \Lambda =s_i \Lambda$.  In particular, $ s^*_\Lambda(x;t)$ is symmetric in $x_i,x_{i+1}$ if $a_i=a_{i+1}$.
\end{proposition}  
\begin{proof}
  The case $a_i > a_{i+1}$ follows immediately from  \eqref{recursis}.
  In the case $a_i < a_{i+1}$, we have from  $T_i s_{\tilde \Lambda}^*=s_\Lambda^*$ and the quadratic relation satisfied by the generators
 of the  Hecke algebra that
  $$
T_i s_\Lambda^* =  T_i^2 s_{\tilde \Lambda}^*= ((t-1) T_i +t) s_{\tilde \Lambda}^*
  $$
We thus obtain, using again $T_i s_{\tilde \Lambda}^*=s_\Lambda^*$, that
$$
T_i s_\Lambda^* = (t-1) s_{\Lambda}^* + t s_{\tilde \Lambda}^*
  $$
as wanted.

The case $a_i=a_{i+1}$ is somewhat less trivial.
First consider the case when $\pmb a$ is dominant in which $s_\Lambda^*(x;t)$ is a multi-Schur function.  From
the multi-Schur property \eqref{idenmS}, we have that
$s_\Lambda^*(x;t)$ is symmetric in $x_i$ and $x_{i+1}$ whenever $a_i=a_{i+1}$. By Remark~\ref{remarksym}, we therefore have $T_i s_\Lambda^*(x;t)=t s_\Lambda^*(x;t)$ in that case.

Finally, suppose that $\pmb a$ is not dominant and that $a_i=a_{i+1}$.  We will proceed by induction on the number of inversions of $\pmb a$,
knowing that we have established the case when the number of inversions is zero.  First
suppose  that $j$ is such that $\{ j,j+1 \}$ does not intersect $\{i,i+1\}$ and such that $s_j \pmb a$ has one fewer inversion than $\pmb a$.
Using the formula $T_j s_{s_j \Lambda}^*= s_\Lambda^*$ that we have already shown to hold, we then have 
by induction that
$$
T_i  s_\Lambda^* =  T_i T_j( s_{s_j \Lambda}^*)= T_j (T_i s_{s_j \Lambda}^*)=t T_j s_{s_j \Lambda}^*=t s_\Lambda^*
$$
Now suppose that $s_{i-1} \pmb a$ has one fewer inversion than $\pmb a$.  Then $s_i s_{i-1} \pmb a$ will  have two fewer inversions  than $\pmb a$ since $a_i=a_{i+1}$.  Hence, using this time
$T_{i-1} T_i s_{s_i s_{i-1} \pmb a}^*=  s_\Lambda^*$, we have by induction that
$$
T_i  s_\Lambda^* = T_i (T_{i-1} T_i s_{s_i s_{i-1} \pmb a}^*)= T_{i-1} T_{i} (T_{i-1} s^*_{s_i s_{i-1} \pmb a})=t T_{i-1} T_{i} s_{s_i s_{i-1} \pmb a}^*= t s_\Lambda^*
$$
where we have used the braid relations obeyed by the Hecke algebra generators.
Finally, if $s_{i+1} \pmb a$ has one fewer inversion than $\pmb a$,  then
$s_i s_{i+1} \pmb a$  has again two fewer inversions than $\pmb a$
and we proceed as in the previous case.
\end{proof}
The analog of Proposition~\ref{propoTstar} holds for the 
 $m$-symmetric Schur functions.
\begin{corollary} \label{coroTons}
  For $i=1,\dots,m-1$, the operator $T_i$ is such that
$$
T_i  s_\Lambda =
\left \{ \begin{array}{ll}
     s_{\tilde \Lambda}  & {\rm if~} a_i > a_{i+1} \\
    (t-1)  s_{\Lambda} +  t s_{\tilde \Lambda}  & {\rm if~} a_i < a_{i+1} \\
     t s_{\Lambda}  & {\rm if~} a_i = a_{i+1}
\end{array} \right .    
$$
where $\tilde \Lambda=s_i \Lambda$. In particular, $ s_\Lambda(x;t)$ is symmetric in $x_i,x_{i+1}$ if $a_i=a_{i+1}$.
\end{corollary}
\begin{proof}
Suppose that $a_i > a_{i+1}$. We will compute $\langle T_i s_\Lambda, s_{\Omega}^* \rangle$ for all $\Omega$.
From Definition~\ref{defmschur} and Proposition~\ref{propoTidual}, we get
that $\langle T_i s_\Lambda, s_{\Omega}^* \rangle = \langle s_\Lambda,  T_i s_{\Omega}^* \rangle = 0$ whenever  $\Omega\not \in \{ \Lambda,\tilde \Lambda\}$. When $\Omega \in \{ \Lambda,\tilde \Lambda\}$, we have again from  Definition~\ref{defmschur} and Proposition~\ref{propoTidual} that
$$
\langle T_i s_\Lambda, s_{\Lambda}^* \rangle = \langle  s_\Lambda, T_i s_{\Lambda}^* \rangle = \langle  s_\Lambda, s_{\tilde \Lambda}^* \rangle =0
$$
and that
$$
\langle T_i s_{ \Lambda}, s_{\tilde \Lambda}^* \rangle = \langle  s_\Lambda, T_i s_{\tilde \Lambda}^* \rangle = \langle  s_\Lambda, t s_{\Lambda}^* \rangle = t^{1+{\rm Inv}(\pmb a)}= t^{{\rm Inv}(s_i \pmb a)}
$$
from which we deduce that $T_i s_\Lambda = s_{\tilde \Lambda}$.  The other cases can be checked similarly. 
\end{proof}  
From Proposition~\ref{propall0}, we know explicitly $s_\Lambda^*(x;t)$ when $\Lambda=(0,0,\dots,0; \lambda)$.
There is also a family of  $m$-symmetric Schur functions for which we have explicit expressions.
\begin{proposition}  \label{propoHspecial}
  If $\Omega=(\pmb b; \emptyset)$ then $s_\Omega(x;t)=k_\Omega(x;t)=H_{\pmb b}(x_1,\dots,x_m;t)$.
\end{proposition}  
\begin{proof} 
From Proposition~\ref{propdualtriang}, the coefficients $D_{\Lambda \Omega}(t)$ in \eqref{eqDs} are equal to zero if $|\pmb b| \geq |\pmb a|$. We thus have that the
coefficients $D_{\Lambda \Omega}(t)$ in \eqref{eqDs} are all equal to zero when $\Omega=(\pmb b;\emptyset)$, by the maximality of $\pmb b$. From \eqref{eqDs}, this yields $k_\Omega(x;t)=s_\Omega(x;t)$ as wanted.
\end{proof}  

Since an $m$-symmetric function is also an $(m+1)$-symmetric function, it is natural to consider the inclusion $i: R_m \to R_{m+1}$. 
It is straightforward to verify that
$$
i(p_\Omega(x)) =p_{\Omega^0}(x)  \qquad {\rm and } \qquad i(k_\Omega(x;t)) =k_{\Omega^0}(x;t) 
$$
where 
$\Omega^0=(\pmb b^0;\mu)$ with $\pmb b^0=(b_1,\dots,b_m,0)$. The restriction $r: R_{m+1} \to R_m$ is defined on the other hand as
$$
r(f)=f(x_1,\dots,x_m,0,x_{m+2},x_{m+3},\dots)  \big |_{(x_{m+2},x_{m+3},\dots) \mapsto (x_{m+1},x_{m+2},\dots) }
$$
It is again straightforward to check that
\begin{equation} \label{rH}
r(p_\Omega(x)) =
\left \{ \begin{array}{ll}
  p_{\Omega_-}(x) & {\rm if~} b_{m+1}=0 \\
  0 & {\rm otherwise}
\end{array} \right .  
\qquad {\rm and} \qquad
r(k_\Omega(x;t)) =
\left \{ \begin{array}{ll}
  k_{\Omega_-}(x;t) & {\rm if~} b_{m+1}=0 \\
  0 & {\rm otherwise}
\end{array} \right .
\end{equation}
where 
$\Omega_-=(\pmb b_-;\mu)$ with $\pmb b_-=(b_1,\dots,b_m)$.
It is then immediate that $r \circ i : R_m \to R_m$ is the identity and that
\begin{equation} \label{dualir}
\langle i(f),g \rangle_{m+1} = \langle f, r(g) \rangle_m
\end{equation}
for all $f\in R_m$ and all $g \in R_{m+1}$.
Remarkably, the inclusion and restriction of the (dual) $m$-symmetric Schur functions satisfy very elegant formulas.
\begin{proposition} For  $\Lambda =(a_1,\dots,a_m; \lambda )$, 
we have that
\begin{equation} \label{inclust}
r(s_\Lambda^*)= t^{{\rm Inv}(\pmb a)-{\rm Inv} (\hat {\pmb a}) }s_{\hat \Lambda}^*
\end{equation}
where $\hat \Lambda =(a_1,\dots,a_{m-1}; \lambda \cup (a_{m}))$.  Equivalently, by duality and \eqref{dualir}, 
\begin{equation} \label{inclus}
i(s_\Lambda) = \sum_\Omega s_\Omega
\end{equation}
where the sum is over all  $\Omega$'s such that $\Omega=(a_1,\dots,a_m,a_{m+1};\mu)$ with $\mu \cup (a_{m+1})=\lambda$.

In other words, the $m$-partition $\hat \Lambda$ in \eqref{inclust} is obtained from $\Lambda$ by removing the $m$-circle, while the $\Omega$'s that appear in \eqref{inclus} are those whose diagram can be obtained from that of $\Lambda$ by adding an $(m+1)$-circle in any symmetric row (including that of size 0).
\end{proposition}
Before proving the proposition, we illustrate it with an example.
\begin{example} The diagram of  $\Lambda=(1,3;4,3,2)$ is 
  \begin{equation*}
 {\small{\tableau[scY]{&& &  \\& & & \bl \cercle{\rm 2}  \\& &  \\ & \\ & \bl \cercle{\rm 1}   \\}}}
  \end{equation*}
  Adding a 3-circle  in all possible ways gives
 \begin{equation*}
 {\small{\tableau[scY]{&& & &  \bl \cercle{\rm 3} \\& & & \bl \cercle{\rm 2}  \\& &  \\ & \\ & \bl \cercle{\rm 1}   \\}}} \qquad   \qquad  {\small{\tableau[scY]{&& &  \\& & & \bl \cercle{\rm 2}  \\& & &  \bl \cercle{\rm 3}\\ & \\ & \bl \cercle{\rm 1}   \\}}}  \qquad   \qquad  {\small{\tableau[scY]{&& &  \\& & & \bl \cercle{\rm 2}  \\& &  \\ & &  \bl \cercle{\rm 3} \\ & \bl \cercle{\rm 1}   \\}}}
\qquad  \qquad   {\small{\tableau[scY]{&& &  \\& & & \bl \cercle{\rm 2}  \\& &  \\ & \\ & \bl \cercle{\rm 1} \\ \bl \cercle{\rm 3}  \\}}}
 \end{equation*}
 Hence
 $$
i(s_{1,3;4,3,2})= s_{1,3,4;3,2} + s_{1,3,3;4,2} + s_{1,3,2;4,3} + s_{1,3,0;4,3,2}
 $$
  \end{example}  
\begin{proof}
  Since $T_\sigma$ commutes with the restriction $r:R_m \to R_{m-1}$ when $\sigma \in S_{m-1}$,  
  we can suppose that
$\hat {\pmb a}=(a_1,\dots,a_{m-1})$ is dominant.  In this case, 
  ${\pmb a'}=(a_1,\dots,a_{\ell-1},a_{m},a_\ell,\dots, a_{m-1})$ is also dominant, where
$a_{m}$ is such that  $a_{\ell-1} \geq a_{m}>a_\ell$.
 Thus
  $$
s_{\Lambda}^* = T_{m-1} \cdots T_{\ell} s_{\Lambda'}^*
  $$
where $\Lambda'=({\pmb a'};\lambda)$.  It is easy to deduce from \eqref{eqTi}  that
\begin{equation}
 r\bigl( T_{m-1}  f \bigr)= t f(x_1,\dots,x_{m-2},0,x_{m},x_{m+1},\dots)  \big |_{(x_{m},x_{m+1},\dots) \mapsto (x_{m-1},x_{m},\dots) }
\end{equation}  
on any $f\in R_m$, and more generally, that
\begin{equation}
 r\bigl( T_{m-1} \cdots T_\ell   f \bigr)= t^{m-\ell} f(x_1,\dots,x_{\ell-1},0,x_{\ell+1},x_{\ell+2},\dots)  \big |_{(x_{\ell+1},x_{\ell+2},\dots) \mapsto (x_{\ell},x_{\ell+1},\dots) }
\end{equation}  
Hence, given that $\Lambda'$ is dominant,
 $$
r(s_{\Lambda}^*) = r(T_{m-1} \cdots T_{\ell} s_{\Lambda'}^*)
  $$
is (up to a $t$ power) obtained by setting $x_\ell=0$ in the multi-Schur function corresponding to $s_{\Lambda'}^*$ followed by 
relabeling the variables $x_{\ell+1},x_{\ell+2},\dots$. But this simply  generates (up to a $t$ power)  the multi-Schur function associated to $s_{\hat \Lambda}^*$,
where $\hat \Lambda =(a_1,\dots,a_{m-1}; \lambda \cup (a_{m}))$. Observing that $m-\ell={\rm Inv}(\pmb a)-{\rm Inv} (\hat {\pmb a})$,  \eqref{inclust} is seen to hold.

\end{proof}

\begin{proposition} \label{propoinclusion0} For $\Lambda=(a_1,\dots,a_{m};\lambda)$, let
 $\Lambda^0=(a_1,\dots,a_{m},0;\lambda)$.  We have that
  $$
i(s^*_\Lambda)=s^*_{\Lambda^0}
$$
Equivalently, by duality,
$$
r(s_{\Lambda}) =
\left \{ \begin{array}{ll}
  0 & {\rm if~} a_{m}>0 \\
  s_{\Lambda_-} & {\rm if~}a_{m}=0
\end{array}   \right .
$$
where $\Lambda_-=(a_1,\dots,a_{m-1};\lambda)$.
\end{proposition}  
\begin{proof}
 Since $T_\sigma$ commutes with the inclusion $i$ when $\sigma \in S_m$,
  we can suppose that
$\Lambda$ is dominant.  In this case $i(s^*_\Lambda)=s^*_{\Lambda^0}$ is immediate from Proposition~\ref{proprecursion}. 
\end{proof}

\section{Main conjecture} \label{secMain}

Our main conjecture is a positivity conjecture for the expansion of $m$-symmetric Macdonald polynomials in terms of $m$-symmetric Schur functions akin to the original Macdonald positivity conjecture \cite{M} (now theorem \cite{Haiman}).  Before stating the conjecture, we need to extend the notion of plethysm to $R_m$.
Recall that the plethysm relevant to Macdonald polynomials is the linear map on the ring of symmetric functions that sends the power-sum 
$p_\lambda$ to $p_\lambda /\prod_i (1-t^{\lambda_i})$ \cite{Ber}.  The notion of plethysm that we will need will simply be the linear map $\wp:R_m \to R_m$  defined on the $m$-symmetric power-sums as
\begin{equation} \label{plethysm}
\wp(p_\Lambda) = \frac{p_\Lambda }{\prod_i (1-t^{\lambda_i})}
\end{equation}
We stress that the plethysm $\wp$ only depends on the symmetric part $\lambda$ of $\Lambda=(\pmb a; \lambda)$.

We now introduce the proper normalization for the integral form of the $m$-symmetric Macdonald polynomial. Let
$$
c_\Lambda(q,t)= \prod_{s \in \Lambda} (1 -q^{a(s)}t^{\ell(s)+1})
$$
where the product is over the cells of $\Lambda$ (excluding the circles).
In the expression, $a(s)$ and $\ell(s)$ stand respectively  for the arm-length and leg-length of $s$. The arm-length, $a(s)$, 
is equal to the number of cells in $\Lambda$ strictly to the right of $s$ (and in the same row).  Note that if there is a circle at the end of its row, then it adds one to the arm-length of $s$. The leg-length, $\ell(s)$, is equal to the number of cells in $\Lambda$ strictly below $s$ (and in the same column). If at the bottom of its column there are $k$ circles whose filling is smaller than the filling of the circle at the end of its row, then they add $k$ to the value of the leg-length of $s$.  If the row does not end with a circle then none of the circles at the bottom of its column  contributes to the leg-length.
\begin{example}
The values of $a(s)$ and $\ell(s)$ in each cell of the diagram of  $\Lambda={(2,0,0,2;4,1,1)}$ are
$$
 {\tableau[scY]{\mbox{\scriptsize{\rm 34}}&\mbox{\scriptsize{\rm 22}}&  \mbox{\scriptsize{\rm 10}} & \mbox{\scriptsize{\rm 00}} & \bl \\\mbox{\scriptsize{\rm 23}}& \mbox{\scriptsize{\rm 11}}& \bl \cercle{\rm 1} \\\mbox{\scriptsize{\rm 24}}& \mbox{\scriptsize{\rm 10}} & \bl \cercle{\rm 4}\\ \mbox{\scriptsize{\rm 01}}& \bl  \\ \mbox{\scriptsize{\rm 00}} \\ \bl \cercle{\rm 2} \\ \bl \cercle{\rm 3}}}
 $$
which gives
$$
c_{(2,0,0,2;4,1,1)}(q,t)= (1-t)(1-qt)(1-q^2t^3)(1-q^3t^5)(1-qt^2)(1-q^2t^4)(1-qt)(1-q^2t^5)(1-t^2)(1-t) 
$$
\end{example}
Let the integral form of the $m$-symmetric Macdonald polynomials be defined as\footnote{The integral form of the ``partially symmetric Macdonald polynomials'' in \cite{BG,BG2} is the same as our integral form.}
$$
J_\Lambda(x;q,t)= c_\Lambda(q,t) P_\Lambda(x;q,t)
$$
We are finally ready to state our positivity conjecture for the $m$-Macdonald polynomials.  
\begin{conjecture} \label{mainconjec}
  The $m$-symmetric Macdonald polynomials in their integral form are such that
  \begin{equation} \label{Kostka}
  \wp\bigl(J_\Lambda(x;q,t)\bigr)=
  \sum_{\Omega} K_{\Omega \Lambda}(q,t) \, s_\Omega(x;t)
  \end{equation}
with $K_{\Omega \Lambda}(q,t) \in \mathbb N[q,t]$.
\end{conjecture}
It is noteworthy that, as we will see in Corollary~\ref{corospecial},  the usual $(q,t)$-Kostka coefficients $K_{\mu \lambda}(q,t)$ are special cases of the coefficients $K_{\Omega \Lambda}(q,t)$.
 \begin{example} \label{expos}We have
\begin{align*}
  \wp (J_{1,0;2})& = t^2 s_{3,0;\emptyset} + qt^2 s_{0,3;\emptyset}+ q t \, s_{0,0;3} +(qt^3+t) s_{2,1;\emptyset}+(qt^2+t) s_{2,0;1}
  +(q^2t^3+t) s_{1,2;\emptyset} \\  & \qquad +(q^2t^2+1) s_{1,0;2}+(q^2t^2+qt) s_{0,2;1}+(q^2t^2+q) s_{0,1;2}+(q^2t+q) s_{0,0;2,1} \\
  & \qquad \qquad + qt^2 s_{1,1;1}+ qt \, s_{1,0;1,1} + q^2 t \, s_{0,1;1,1}+q^2s_{0,0;1,1,1}
\end{align*}
\end{example} 
 \begin{remark}
   The symmetries  $K_{\mu \lambda}(q,t)=K_{\mu' \lambda'}(t,q)$ and $K_{\mu \lambda}(q,t)=q^{n(\lambda')}t^{n(\lambda)} K_{\mu' \lambda}(q^{-1},t^{-1})$ (where $n(\lambda)=\lambda_2+2\lambda_3+3\lambda_4+\cdots$) of the usual $(q,t)$-Kostka coefficients do not extend to the $m$-symmetric world (except in the case $m=1$). It seems however that  $K_{\Omega \Lambda}(q,t)=K_{\Omega' \Lambda'}(t,q)$ whenever $\Omega'$ and $\Lambda'$ both exist, where $\Lambda'$ is the $m$-partition whose diagram is the conjugate of the diagram of $\Lambda$.
\end{remark}
\begin{remark} \label{remarkstandard}
The special case $t=1$  will be proven in \cite{CL} using a combinatorial interpretation quite similar to the major index statistic on tableaux (note that this does not provide a formula for the $q=1$ case since, as commented in the previous remark,
  the symmetry  $K_{\mu \lambda}(q,t)=K_{\mu' \lambda'}(t,q)$ of the usual $(q,t)$-Kostka coefficients does not extend to the $m$-symmetric world).  One important consequence of the combinatorial interpretation at $t=1$ worth mentioning is that, as can be appreciated in Example~\ref{expos}, 
$$
K_{\Omega \Lambda}(1,1) = \# \text{ of standard tableaux of shape } \mu \cup \pmb b
$$
Finding a combinatorial interpretation for the  coefficients $K_{\Omega \Lambda}(q,t)$
thus amounts to finding a statistic on standard tableaux with extra circles.
\end{remark}

The next section will be devoted to proving properties of the $K_{\Omega \Lambda}(q,t)$ coefficients.

\section{The $K_{\Omega \Lambda}(q,t)$ coefficients} \label{secKostkaprop}
Before describing a few elementary relations satisfied by the $K_{\Omega \Lambda}(q,t)$ coefficients, we first need to establish certain results on the inclusion and restriction of $m$-symmetric Macdonald polynomials.  The proofs will rely on properties of the $m$-symmetric Macdonald polynomials that are proven in Appendix~\ref{secAppendix}.
We start with the restriction, which  can be obtained in general and which is surprisingly simple in the integral form.

\begin{proposition} \label{proporJ}   For $\Lambda=(a_1,\dots,a_m,a_{m+1} ; \lambda )$,
  the restriction $r:R_{m+1}\to R_m$ is such that
$$  
r ( J_{\Lambda}(x;q,t)) = q^{a_{m+1}} t^{\# \{ i \, | \, a_i < a_{m+1}\} }  J_{ \hat \Lambda} (x;q,t)
$$
where $\hat \Lambda=(a_1,\dots,a_m;\lambda \cup (a_{m+1}))$.
\end{proposition}
\begin{proof}  First suppose that $a_{m+1}=0$.  Letting $N\to \infty$ in Proposition~\ref{propMacdo00}, we have in this case that 
  $$
J_\Lambda(x;q,t) \Big |_{x_{m+1}=0} = J_{\hat \Lambda}(x_{(m+1)};q,t) 
$$
where $x_{(m+1)}=(x_1,\dots,x_m,x_{m+2},x_{m+3},\dots)$.
After relabelling the variables as $(x_{m+2},x_{m+3},\dots) \mapsto (x_{m+1},x_{m+2},\dots)$ we obtain $r ( J_{\Lambda}(x;q,t)) =  J_{ \hat \Lambda} (x;q,t)$ as wanted.

If $a_{m+1}\neq 0$, letting this time $N\to \infty$ in Proposition~\ref{propMacdo0}, we get that
$$
J_\Lambda(x;q,t) \Big |_{x_{m+1}=0} = q^{a_{m+1}} t^{\# \{ i \, | \, a_i < a_{m+1}\} }  J_{\hat \Lambda}(x_{(m+1)};q,t) 
$$
which gives that $r(J_\Lambda(x;q,t)) = q^{a_{m+1}} t^{\# \{ i \, | \, a_i < a_{m+1}\} }  J_{\hat \Lambda}(x;q,t)$ after relabelling the variables. 
\end{proof}  
The inclusion of an $m$-symmetric  Macdonald polynomial is also quite simple when the symmetric part of the indexing $m$-partition is empty.
\begin{proposition} \label{propoiJ}
If $\Lambda=(a_1,\dots,a_m;\emptyset)$, then the inclusion $i:R_m \to R_{m+1}$ is such that
  $$
i(J_\Lambda(x;q,t))= J_{\Lambda^0}(x;q,t) 
$$
where $\Lambda^0=(a_1,\dots,a_m,0;\emptyset)$.
\end{proposition}  
\begin{proof}
  We need to prove that  $i(c_{\Lambda}(q,t) P_{\Lambda}(x;q,t))=c_{\Lambda^0}(q,t) P_{\Lambda^0}(x;q,t)$. It is immediate that $c_{\Lambda^0}(q,t)=c_{\Lambda}(q,t)$ since the only difference between the two diagrams is the extra circle $m+1$ at the bottom of the diagram of $\Lambda^0$. But this circle $m+1$ is larger than every other circle and therefore never contributes to a leg-length.  We thus have to prove that  $i(P_{\Lambda}(x;q,t))= P_{\Lambda^0}(x;q,t)$.  We will show that
  $P_{\Lambda^0}(x;q,t)$ and $P_{\Lambda}(x;q,t)$ are equal as polynomials in $N$ variables, which will prove the proposition.

  We have
$$
P_{\Lambda}(x;q,t)=\frac{1}{u_{\Lambda}}\mathcal S_{m+1,N}^t E_{a_1,\dots,a_m,0^{N-m}}= 
\frac{1}{u_{\Lambda}}\mathcal S_{m+2,N}^t \mathcal O_{m+1} E_{a_1,\dots,a_m,0^{N-m}}
$$
where
$$\mathcal O_{m+1} = 1+T_{m+1}+T_{m+1}T_{m+2} + \cdots+ T_{m+1} T_{m+2} \cdots T_{N-1}
$$  Since $E_{a_1,\dots,a_m,0^{N-m}}$ is symmetric in the variables $x_{m+1},\dots, x_{N}$, we can easily check that
$\mathcal O_{m+1} E_{a_1,\dots,a_m,0^{N-m}}=(1-t^{N-m})/(1-t) E_{a_1,\dots,a_m,0^{N-m}}$, which implies that
$$
P_{\Lambda}(x;q,t)= 
\frac{1}{u_{\Lambda}(t)}\left( \frac{1-t^{N-m}}{1-t}\right)\mathcal S_{m+2}^t E_{a_1,\dots,a_m,0^{N-m}} = \frac{u_{\Lambda^0}(t)}{u_{\Lambda}(t)}\left( \frac{1-t^{N-m}}{1-t}\right) P_{\Lambda^0}(x;q,t)
$$
But $u_{\Lambda}(t)/u_{\Lambda^0}(t)=(1-t^{N-m})/(1-t)$ due to the extra 0 in the symmetric entries of $\Lambda$, which gives that, in $N$ variables, $P_{\Lambda^0}(x;q,t)=P_{\Lambda}(x;q,t)$.
\end{proof}

In order to better visualize the following theorem, it will prove convenient to denote, for $\Lambda=(\pmb a;\lambda)$ and $\Omega=(\pmb b;\mu)$,
the coefficient $K_{\Omega \Lambda} (q,t)$ as
$$
\left( \begin{array}{cccccc}
  a_1 & a_2 & \cdots & a_m & \vline &\lambda \\
   b_1 & b_2 & \cdots & b_m & \vline & \mu
\end{array}
\right)
$$
Recall that
$\lambda\setminus \lambda_j$ stand for the partition obtained by removing the entry $\lambda_j$ from the partition $\lambda$.
\begin{theorem} \label{propokostka}
  The  coefficients  $K_{\Omega \Lambda} (q,t)$ satisfy the following relations.
\begin{enumerate}
\item For any entry $\lambda_j$ in $\lambda$, we have
  $$
 q^{\lambda_j}t^{\# \{ i \, | \,  a_i < \lambda_j \}}
\left( \begin{array}{cccccc}
  a_1 & a_2 & \cdots & a_m & \vline &\lambda \\
   b_1 & b_2 & \cdots & b_m & \vline & \mu
\end{array}
\right)= \left( \begin{array}{ccccccc}
  a_1 & a_2 & \cdots & a_m & \lambda_j & \vline &\lambda \setminus \lambda_j \\
   b_1 & b_2 & \cdots & b_m & 0 & \vline & \mu
\end{array}
\right)
$$
\item
 $$
\left( \begin{array}{cccccc}
  a_1 & a_2 & \cdots & a_m & \vline &\lambda \\
   b_1 & b_2 & \cdots & b_m & \vline & \mu
\end{array}
\right)= \left( \begin{array}{ccccccc}
  a_1 & a_2 & \cdots & a_m & 0 & \vline &\lambda  \\
   b_1 & b_2 & \cdots & b_m & 0 & \vline & \mu
\end{array}
\right)
$$
\item For any entry $\mu_j$ in $\mu$, we have
$$
\left( \begin{array}{cccccc}
  a_1 & a_2 & \cdots & a_m & \vline &\emptyset \\
   b_1 & b_2 & \cdots & b_m & \vline & \mu 
\end{array}
\right)= \left( \begin{array}{ccccccc}
  a_1 & a_2 & \cdots & a_m & 0 & \vline &\emptyset \\
   b_1 & b_2 & \cdots & b_m & \mu_j & \vline & \mu \setminus \mu_j
\end{array}
\right)
$$
\item If $a_{i+1}=a_i$ then
$$
\left( \begin{array}{cccccccc}
  a_1 & \cdots& a_i & a_i & \cdots & a_m & \vline &\lambda \\
   b_1 & \cdots& b_i & b_{i+1} & \cdots & b_m & \vline & \mu
\end{array}
\right)= \left( \begin{array}{cccccccc}
  a_1 & \cdots& a_i & a_i & \cdots & a_m & \vline &\lambda \\
   b_1 & \cdots& b_{i+1} & b_{i} & \cdots & b_m & \vline & \mu
\end{array}
\right) 
$$
\item If $b_{i+1}=b_i$ and $a_i > a_{i+1}$ then
$$
t \left( \begin{array}{cccccccc}
  a_1 & \cdots& a_{i} & a_{i+1} & \cdots & a_m & \vline &\lambda \\
   b_1 & \cdots& b_i & b_{i} & \cdots & b_m & \vline & \mu
\end{array}
\right)= \left( \begin{array}{cccccccc}
  a_1 & \cdots& a_{i+1} & a_i & \cdots & a_m & \vline &\lambda \\
   b_1 & \cdots& b_{i} & b_{i} & \cdots & b_m & \vline & \mu
\end{array}
\right) 
$$
\end{enumerate}
\end{theorem}
\begin{proof}
  We will see that (1) is an easy consequence of Propositions~\ref{propoinclusion0} and \ref{proporJ}.  Applying $r$ on both sides of \eqref{Kostka}, we get
  $$
r (\wp (J_{\Lambda^\diamond}(x;q,t))= \sum_{\Delta} K_{\Delta \Lambda^\diamond} r(s_\Delta (x;t))
$$
where  $\Lambda^\diamond=(a_1,\dots,a_m,\lambda_1; \lambda \setminus \lambda_1)$.
Since the plethysm $\wp$ commutes with $r$, we have from Propositions~\ref{proporJ} that
\begin{equation}  \label{eqr1}
r (\wp (J_{\Lambda^\diamond}(x;q,t))=q^{\lambda_1}t^{\# \{ i \, | \,  a_i < \lambda_1 \}} \wp (J_{\Lambda}(x;q,t)
= q^{\lambda_1}t^{\# \{ i \, | \,  a_i < \lambda_1 \}}  \sum_{\Omega} K_{\Omega \Lambda} s_\Omega (x;t)
\end{equation}
and from Proposition~\ref{propoinclusion0} that
\begin{equation} \label{eqr2}
 \sum_{\Delta} K_{\Delta \Lambda^\diamond} r(s_\Delta (x;t))
= \sum_{\Omega} K_{\Omega^0 \Lambda^\diamond} s_{\Omega} (x;t)
\end{equation}
where $\Omega^0=(b_1,\dots,b_m,0;\mu)$.
Equating \eqref{eqr1} and \eqref{eqr2}, we obtain that
$q^{\lambda_1}t^{\# \{ i \, | \,  a_i < \lambda_1 \}} K_{\Omega \Lambda}(q,t)= K_{\Omega^0 \Lambda^\diamond}(q,t)$ as wanted.

Statement (2) is statement (1) with $\lambda_j=0$.  The proof is the same as the proof of statement (1).

For (3), let $\Lambda=(a_1,\dots,a_m;\emptyset)$.   Using  Proposition~\ref{propoiJ} and the commutativity of $i$ and $\wp$, we have
$$
\sum_\Gamma K_{\Gamma \Lambda}(q,t) i(s_\Gamma(x;t))=i(\wp(J_\Lambda(x;q,t))=\wp(J_{\Lambda^0}(x;q,t))=\sum_\Omega K_{\Omega \Lambda^0}(q,t) s_\Omega(x;t) 
$$
We thus get from \eqref{inclus} that
$$
\sum_\Delta K_{\Gamma \Lambda}(q,t) s_\Delta(x;t)=\sum_\Omega K_{\Omega \Lambda^0}(q,t) s_\Omega(x;t) 
$$
where the sum is over all $\Delta$'s that can be obtained from $\Gamma$ by adding a circle $m+1$ on a symmetric entry. We thus have that $K_{\Omega \Lambda^0} (q,t)= K_{\Gamma \Lambda} (q,t)$ if $\Omega$ can be obtained from $\Gamma$ by adding a circle $m+1$ on a symmetric entry (including a zero entry). But this is exactly the statement of (3) when the entry $\mu_i$ is not equal to zero.

In (4),  $J_\Lambda(x;q,t)$ is in this case symmetric in $x_{i}$ and $x_{i+1}$ from Remark~\ref{remarksymMac}, and so is $\wp(J_\Lambda(x;q,t))$.  
Hence
\begin{equation} \label{eq3}
T_i(\wp(J_\Lambda(x;q,t)))=t\wp(J_{\Lambda}(x;q,t))=t\left (\sum_{\Omega} K_{\Omega \Lambda}(q,t) s_{\Omega}(x;t) \right) 
\end{equation}
On the other hand,
\begin{equation} \label{eq4}
T_i(\wp(J_\Lambda(x;q,t)))= T_i\left (\sum_{\Omega} K_{\Omega \Lambda}(q,t) s_{\Omega}(x;t) \right) =  \sum_{\Omega} K_{\Omega\Lambda}(q,t) T_i s_{\Omega}(x;t) 
\end{equation}
Suppose without loss of generality that $\Omega=(\pmb b; \mu)$ is such that
$b_i>b_{i+1}$ and let $\tilde \Omega=s_i \Omega$.  We then have from Corollary~\ref{coroTons} that
$T_i s_\Omega= s_{\tilde \Omega}$ and $T_i s_{\tilde \Omega}= (t-1)s_{\tilde \Omega}+t s_\Omega$.  Thus, comparing the terms involving $s_\Omega$ and $s_{\tilde \Omega}$ in \eqref{eq3} and \eqref{eq4}, we obtain
$$
t K_{\Omega \Lambda}(q,t) s_{\Omega}(x;t) + t K_{\tilde \Omega \Lambda}(q,t) s_{\tilde \Omega}(x;t)
= K_{\Omega \Lambda}(q,t) s_{\tilde \Omega}(x;t) +  K_{\tilde \Omega \Lambda}(q,t) (t-1)s_{\tilde \Omega}(x;t) + t  K_{\tilde \Omega \Lambda}(q,t) s_{\Omega}(x;t)
$$
which, by linear independence of $s_\Omega$ and  $s_{\tilde \Omega}$, holds iff $ K_{\tilde \Omega \Lambda}(q,t)= K_{\Omega \Lambda}(q,t)$. 

Finally, we consider (5).  Let $\tilde \Lambda=s_i \Lambda$.  We have that $c_\Lambda(q,t)$ and $c_{\tilde \Lambda}(q,t)$ only differ by the circles $i$ and $i+1$ that are switched.  Because $i$ and $i+1$ are consecutive numbers, all the cells in $\Lambda$ and $\tilde \Lambda$ will contribute the same value to $c_\Lambda(q,t)$ and $c_{\tilde \Lambda}(q,t)$ except the cell at the intersection of the two circles.  This cell will contribute a factor of $(1-q^a t^{\ell+1})$ in $c_{\Lambda}(q,t)$ while in  $c_{\tilde \Lambda}(q,t)$ it will contribute a factor of   $(1-q^a t^{\ell+2})$ (given that the circle $i$ is below the circle $i+1$ in $\tilde \Lambda$), where the arm-length $a$ and the leg-length $\ell$ are relative to $\Lambda$.  
Using \eqref{property2} with $x=(x_1,\dots,x_N)$, we also have in our case that
$$
T_i P_\Lambda(x;q,t) = \left(\frac{t-1}{1-q^{-a}t^{-\ell-1}} \right) P_\Lambda(x;q,t)+ \frac{(1-q^{a}t^{\ell+2})(1-q^{a}t^{\ell})}{(1-q^{a}t^{\ell+1})^2} P_{\tilde \Lambda}(x;q,t)
$$
since $\bar \eta_i/\bar \eta_{i+1}=q^a t^{\ell+1}$ and $u_{\Lambda,N}(t)=u_{\tilde \Lambda,N}(t)$.  Therefore, taking $c_\Lambda(q,t)$ and $c_{\tilde \Lambda}(q,t)$ into account, we get
\begin{align} \nonumber
  T_i J_\Lambda(x;q,t) & = \left(\frac{t-1}{1-q^{-a}t^{-\ell-1}} \right)  J_\Lambda(x;q,t)+ \frac{(1-q^{a}t^{\ell+2})(1-q^{a}t^{\ell})}{(1-q^{a}t^{\ell+1})^2}\frac{ c_\Lambda(q,t)}{ c_{\tilde \Lambda}(q,t)}  J_{\tilde \Lambda}(x;q,t) \\ \label{eqact}
 &  = \left(\frac{t-1}{1-q^{-a}t^{-\ell-1}} \right)  J_\Lambda(x;q,t)+\left( \frac{1-q^{a}t^{\ell}}{1-q^{a}t^{\ell+1}} \right)  J_{\tilde \Lambda}(x;q,t) 
\end{align}
from our previous observation on $c_{\Lambda}(q,t)$ and $c_{\tilde \Lambda}(q,t)$.
After applying the plethysm $\wp$ (which commutes with $T_i$), we obtain
$$
T_i \wp(J_\Lambda(x;q,t))
 = \left(\frac{t-1}{1-q^{-a}t^{-\ell-1}} \right)  \wp(J_\Lambda(x;q,t))+\left( \frac{1-q^{a}t^{\ell}}{1-q^{a}t^{\ell+1}} \right)  \wp(J_{\tilde \Lambda}(x;q,t)) 
$$
Focusing on the terms involving $s_\Omega$ on both sides of the equation then yields
$$
K_{\Omega \Lambda}(q,t) T_i s_\Omega =   \left(\frac{t-1}{1-q^{-a}t^{-\ell-1}} \right) K_{\Omega \Lambda}(q,t) s_\Omega +\left( \frac{1-q^{a}t^{\ell}}{1-q^{a}t^{\ell+1}} \right) K_{\Omega \tilde \Lambda}(q,t) s_\Omega
$$
where we used the fact that $s_\Omega$  cannot appear in any $T_i s_\Gamma$ such that $\Gamma\neq \Omega$. We have from Corollary~\ref{coroTons} that $T_i s_\Omega=t s_\Omega$ since $b_i=b_{i+1}$, which implies that
$$
t K_{\Omega \Lambda}(q,t)  =   \left(\frac{t-1}{1-q^{-a}t^{-\ell-1}} \right) K_{\Omega \Lambda}(q,t)  +\left( \frac{1-q^{a}t^{\ell}}{1-q^{a}t^{\ell+1}} \right) K_{\Omega \tilde \Lambda}(q,t) 
$$
Combining the coefficients of $K_{\Omega \Lambda}(q,t)$, we obtain
$$
t \left( \frac{1-q^{a}t^{\ell}}{1-q^{a}t^{\ell+1}} \right) K_{\Omega \Lambda}(q,t) = \left( \frac{1-q^{a}t^{\ell}}{1-q^{a}t^{\ell+1}} \right) K_{\Omega \tilde \Lambda}(q,t) 
$$
which readily implies our claim.
\end{proof}
We get from Theorem~\ref{propokostka}(2) that the usual $(q,t)$-Kostka coefficients $K_{\mu \lambda}(q,t)$ are special cases of the coefficients $K_{\Omega \Lambda}(q,t)$. 
\begin{corollary} \label{corospecial}
  If $\Lambda=(0^m;\lambda)$ and $\Omega=(0^m;\mu)$ then
  $$
K_{\Omega \Lambda}(q,t)=K_{\mu \lambda}(q,t)
$$
where $K_{\mu \lambda}(q,t)$ is the usual Macdonald $(q,t)$-Kostka coefficient.
\end{corollary}

Let
\begin{equation} \label{eqPsi}
\Psi_N =(1-t) (1+T_{N-1}+T_{N-2} T_{N-1} + \cdots + T_m \cdots T_{N-1}) \Phi_q 
\end{equation}
where we recall that $\Phi_q$ was defined in \eqref{eqPhi}. 
The following theorem, which is the analog of \eqref{eqPhi} in the $m$-symmetric world, is proved in Appendix~\ref{secAppendix}.
\begin{theorem} \label{propom1}
  For any $N$ and any $m \geq 1$, we have that
  $$
\Psi_N J_\Lambda (x_1,\dots,x_N;q,t)= t^{-\# \{ 2 \leq j \leq m \, | \, a_j \leq  a_1 \}} 
J_{\Lambda^\square} (x_1,\dots,x_N;q,t)
$$
where $\Lambda^\square =\bigl((a_2,\dots,a_{m}); \lambda \cup (a_1+1)\bigr)$.  In particular, if $m=1$, we have
 $$
\Psi_N J_{(a_1;\lambda)} (x_1,\dots,x_N;q,t)=  
J_{\lambda \cup (a_1+1)} (x_1,\dots,x_N;q,t)
$$
\end{theorem}  
We can now establish the next  proposition, which provides a decomposition of the usual $(q,t)$-Kostka coefficients in terms of $(q,t)$-coefficients at $m=1$ of lower degree.
\begin{proposition} \label{propoKostka1}
  For any entry $\lambda_j$ in $\lambda$,
  the usual $(q,t)$-Kostka coefficient $K_{\mu \lambda}(q,t)$
is such that
 $$
K_{\mu \lambda}(q,t) = \left( \begin{array}{cc}
   \vline &\lambda \\
    \vline & \mu
\end{array}
 \right)= \sum_{\mu_i}
 \left( \begin{array}{ccc}
\lambda_j-1 & \vline &\lambda \setminus \lambda_j \\
 \mu_i -1  &  \vline & \mu \setminus \mu_i
\end{array}
\right)
$$
where the sum is over all distinct parts $\mu_i$ in $\mu$. The positivity conjecture for $m$-symmetric Macdonald polynomials at $m=1$ is thus a refinement of the usual Macdonald positivity.
\end{proposition}
\begin{example}Taking $\lambda=(3,3)$ and $\mu=(3,2,1)$, the decomposition (which is unique in this case since $\lambda$ has only one distinct entry) gives
  $$
K_{(3,2,1),(3,3)}(q,t)=K_{(2;2,1),(2;3)}(q,t)+ K_{(1;3,1),(2;3)}(q,t)+ K_{(0;3,2),(2;3)}(q,t)
$$
A computer calculation  tells us that
$$
K_{(2;2,1),(2;3)}(q,t)=q^4 t^3 + q^3 t^2 + q^2 t^2 + q^2 t + q t,\quad
K_{(1;3,1),(2;3)}(q,t)= q^5t^3 + q^4 t^2 + q^3 t^2 + q^3 t + q^2 t + q
$$
and
$$
K_{(0;,3,2),(2;3)}(q,t) = q^5 t^2 + q^4 t^2 + q^4 t + q^3 t + q^2
$$
which is consistent with the value of $K_{(3,2,1),(3,3)}(q,t)$ found in \cite{M}.
Observe that the diagram of $(2;3)$ is obtained from that of $\lambda=(3,3)$ by transforming a box into a circle (indexed by a 1).  Similarly, the diagram of $(2;2,1)$, $(1;3,1)$ and $(0;3,2)$ are those that can be obtained from that of $\mu=(3,2,1)$ by transforming a box into a circle.

\end{example}

\begin{proof}
  We work in $N$ variables.  Consider the operator $\tilde \Psi_N=\wp \circ \Psi_N \circ \wp^{-1}$.  From Theorem~\ref{propom1}, we have that $\tilde \Psi_N \circ \wp (J_{(a_1;\lambda)})=\wp ( J_{\lambda \cup (a_1+1)})$. In order to obtain our decomposition, we thus need to obtain the action of $\tilde \Psi_N$ on the $m$-symmetric Schur functions. Before obtaining that action, 
we will first show that $\tilde \Psi_N$ is such that
\begin{equation} \label{eqprimera}
\tilde \Psi_N  \bigl( k_{(b_1;\mu)} (x;t) \bigr) = h_{b_1+1}(x) s_{\mu} (x)
\end{equation}
 From $k_{(b_1;\mu)}(x;t)=x_1^{b_1}s_\mu$, we get
$$
\wp^{-1} \bigl( k_{(b_1;\mu)} (x;t) \bigr) = x_1^{b_1} \wp^{-1} (s_\mu)
$$
Since $\wp^{-1} (s_\mu)$ commutes with $\Psi_N$, we then have that
$$
 \Psi_N \circ \wp^{-1} \bigl( k_{(b_1;\mu)} (x;t) \bigr) =  \wp^{-1} (s_\mu) 
 \Psi_N (x_1^{b_1})
 $$
 Now, in the case $m=1$, the operator $\Psi_N$ is such that
 $$
\Psi_N(x_1^{b_1}) = (1-t) (1+T_{N-1} + T_{N-2} T_{N-1} +\dots+ T_1 T_2 \cdots T_{N-1}) x_1^{b_1+1} =(1-t) P_{(b_1+1)}(x_1,\dots,x_N;t)
 $$
where $P_{(b_1+1)}(x_1,\dots,x_N;t)$ is the (symmetric) Hall-Littlewood polynomial indexed by a single entry \cite{M}.  In that case, it is known \cite{Ber,M} that
$$
(1-t)P_{(b_1+1)}(x_1,\dots,x_N;t) = \wp^{-1}({h_{b_1+1}})
$$
from which we obtain that $\Psi_N(x_1^{b_1})=\wp^{-1}({h_{b_1+1}})$. The map $\wp^{-1}$ being a homomorphism, it then follows that
$$
 \Psi_N \circ \wp^{-1} \bigl( k_{(b_1;\mu)} (x;t) \bigr) =  \wp^{-1} (s_\mu) 
 \Psi_N (x_1^{b_1})= \wp^{-1} (h_{b_1+1} s_\mu) 
  $$
 which implies \eqref{eqprimera}.

 We will now show that
 \begin{equation} \label{segunda}
 \tilde \Psi_N(s_{(a_1;\lambda)})=s_{\lambda \cup (a_1+1)}
 \end{equation}
In the $m=1$ case, every $m$-partition is dominant. We thus obtain from Remark~\ref{remarkm1}  using the duality (see \eqref{eqDs} and \eqref{DDs}) that
$$
k_{(b_1;\mu)}(x;t) = \sum_{(a_1;\lambda)} s_{(a_1;\lambda)}(x;t)
$$
where the sum is over all $(a_1;\lambda)$'s such that $(\lambda \cup (a_1))/\mu$ is a horizontal $b_1$-strip whose entries all lie within the first $a_1$ columns.  Therefore, it follows that
$$
\tilde \Psi_N (k_{(b_1;\mu)}) = h_{b_1+1} s_\mu = \sum_\nu s_\nu =\sum_{(a_1;\lambda)} \tilde \Psi_N (s_{(a_1;\lambda)})
$$
where the first sum is over all $\nu$'s such that $\nu/\mu$ is a horizontal $(b_1+1)$-strip. Since the $k_{(b_1;\mu)}$'s form a basis of $R_1$,
it suffices to prove that
  $\tilde \Psi_N(s_{(a_1;\lambda)})=s_{\lambda \cup (a_1+1)}$ implies that
$\tilde \Psi_N (k_{(b_1;\mu)}) = h_{b_1+1} s_\mu$.  We will see that it is indeed the case.  We have
$$
\sum_{(a_1;\lambda)} \tilde \Psi_N (s_{(a_1;\lambda)})= \sum_{(a_1;\lambda)} s_{\lambda \cup (a_1+1)}
$$
where the sum in the second sum can be thought as being over all $(a_1;\lambda)$'s such that $(\lambda \cup (a_1+1))/\mu$ is a horizontal $(b_1+1)$-strip whose entries all lie within the first $a_1+1$ columns (and including a cell in column $a_1+1$).  This implies that the map $\nu \mapsto \lambda \cup (a_1+1)$, where $a_1+1$ is the rightmost column in $\nu/\mu$,
is an obvious bijection between the $\nu$'s in $\sum_\nu s_\nu$ and the  $(a_1;\lambda)$'s in  $\sum_{(a_1;\lambda)} s_{\lambda \cup (a_1+1)}$. Hence   $\tilde \Psi_N(s_{(a_1;\lambda)})=s_{\lambda \cup (a_1+1)}$ implies
$\tilde \Psi_N (k_{(b_1;\mu)}) = h_{b_1+1} s_\mu$ and \eqref{segunda} holds.

Applying $\tilde \Psi_N$ on both sides of 
$$
\wp (J_{(a_1;\lambda)})= \sum_{(b_1;\mu)} K_{(b_1;\mu) (a_1;\lambda)}(q,t) s_{(b_1;\mu)}
$$
we thus get from Theorem~\ref{propom1} that
$$
\wp (J_{\lambda \cup(a_1+1)}) = \sum_{\nu} K_{\nu, \lambda \cup(a_1+1)}(q,t) s_\nu =
 \sum_{(b_1;\mu)} K_{(b_1;\mu) (a_1;\lambda)}(q,t) s_{\mu \cup (b_1+1) }
$$
Equating the coefficients of $s_\nu$ on both sides, the proposition then follows straightforwardly.
\end{proof}

\section{Non-symmetric Macdonald polynomials as $m$-symmetric functions and
  the large $m$ limit} \label{seclargem}
Suppose that $\eta=(\eta_1,\dots,\eta_\ell,0^{N-\ell})$ is a composition of length $\ell$.
Then 
$$ 
T_i E_\eta(x;q,t)= t E_\eta(x;q,t) 
$$
for all $\ell < i < N$ since $\eta_{i}=\eta_{i+1}=0$ in that case. Hence, $\mathcal S_m^t E_\eta(x;q,t)$ is equal to $E_\eta(x;q,t)$ times a constant whenever $m \geq \ell$. We thus get that  if $\Lambda= (\eta_1,\dots,\eta_\ell,0^{m-\ell}; \emptyset)$ then
$$
P_\Lambda(x;q,t) = E_\eta(x;q,t)
$$
since the coefficient of $x^\eta$ is equal to 1 in both functions.
That is, when $m$ is large enough, a non-symmetric Macdonald polynomial is also an $m$-Macdonald polynomial.
In this case, the integral form of the non-symmetric Macdonald polynomial,
$J_\eta(x;q,t)= c_{(\eta_1,\dots,\eta_\ell ; \emptyset)}(q,t) E_\eta(x;q,t)$, 
has the following expansion when considered as an $m$-symmetric function:
$$
  \wp(J_\eta(x;q,t)) = \sum_\Omega K_{\Omega (\eta_1,\dots,\eta_\ell,0^{m-\ell} ; \emptyset)} (q,t) s_{\Omega}(x;t)
$$

As we have seen in Theorem~\ref{propokostka}, by increasing the value of $m$, one can move to the non-symmetric side the entries on the symmetric side of $\Lambda$ and $\Omega$.  We can also use this idea in the other direction.  Let $\omega$ be a composition of arbitrary length $\ell'$.  Taking $m \geq \max \{ \ell, \ell'\}$,  we can define $K_{\omega \eta}(q,t)$ as
$$
K_{\omega \eta}(q,t)=
\left( \begin{array}{ccccccccccc}
  \eta_1 & \cdots & \eta_\ell& 0&  \cdots &0&  0 & \cdots & 0 &  \vline &\emptyset \\
   \omega_1  & \cdots & \omega_\ell& \omega_{\ell+1}& \cdots & \omega_{\ell'}& 0 & \cdots & 0 & \vline & \emptyset 
\end{array}
\right)
$$
where the number of columns on the non-symmetric side is $m$ (we have supposed  that $\ell' > \ell$ only for the sake of illustrating the general behavior). Note that this definition does not depend on $m$ because we 
can add any number of columns of 0's to the right
without changing the value of $K_{\omega \eta}(q,t)$ by  Theorem~\ref{propokostka}(2).
 Now, from
Theorem~\ref{propokostka}(2)(3)(4), we also have that
$$
K_{\omega \eta} (q,t)=
\left( \begin{array}{ccccccccccc}
  \eta_1 & \cdots & \eta_\ell& 0&  \cdots &0&  0 & \cdots & 0 &  \vline &\emptyset \\
   \omega_1  & \cdots & \omega_\ell& \nu_{1}& \cdots & \nu_{\ell(\nu)} & 0 & \cdots & 0 & \vline & \emptyset 
\end{array}
\right)=\left( \begin{array}{ccccc}
  \eta_1 & \cdots & \eta_\ell & \vline &\emptyset \\
   \omega_1 & \cdots & \omega_\ell & \vline & \nu 
\end{array}
\right)
=K_{\Omega^{\omega} (\eta_1,\dots,\eta_\ell ; \emptyset)}(q,t)  
$$
where $\nu=(\omega_{\ell+1},\dots,\omega_{\ell'})^+$ is the partition obtained by reordering the entries of $(\omega_{\ell+1},\dots,\omega_{\ell'})$, and where $\Omega^{\omega}=(\omega_1,\dots,\omega_\ell;\nu)$.
For a given $\eta$, there are thus only a finite number of distinct coefficients $K_{\omega \eta} (q,t)$ (corresponding to the number of $\ell$-partitions of degree $|\eta|$).  We also have that if $m \geq \ell+|\eta|$, then every $K_{\Omega (\eta_1,\dots,\eta_\ell ; \emptyset)}$ is equal to a given $K_{\omega \eta}(q,t)$ such that $\ell(\omega) \leq m$.  As such, all the information is contained in the  $K_{\omega \eta}(q,t)$'s such that $\ell(\omega) \leq \ell+|\eta|$.

Doing the expansion in the $m$-symmetric world (with  $m \geq \ell+|\eta|$), we
obtain
\begin{equation} \label{expamod}
 \wp( J_\eta(x;q,t)) = \sum_{\omega \in \mathbb Z_{\geq 0}^m} K_{\omega \eta} (q,t) \, s_{(\omega_1,\dots,\omega_m;\emptyset)}(x;t)  \mod  \mathcal L_m
\end{equation}
where $\mathcal L_m$ is the linear span of $m$-symmetric Schur functions indexed by $m$-partitions whose symmetric side is \underline{not} empty. From the discussion above, there is no loss of information in working modulo  $\mathcal L_m$, as the full $m$-symmetric expansion can be recovered from the r.h.s. in \eqref{expamod}. From Proposition~\ref{propoHspecial}, we also get $s_{(\omega_1,\dots,\omega_m;\emptyset)}(x;t)=H_\omega(x;t)$.  We thus have  the following positivity conjecture on non-symmetric Macdonald polynomials.
\begin{conjecture} Let $\eta$ be a composition of length $\ell$.  Then, for $m \geq \ell$, we have that 
$$
 \wp( J_\eta(x;q,t)) = \sum_{\omega \in \mathbb Z_{\geq 0}^m} K_{\omega \eta} (q,t) \, H_\omega(x;t)  \mod  \mathcal L_m
$$
with $K_{\omega \eta} (q,t) \in \mathbb N[q,t]$.
\end{conjecture}
\begin{remark} It is not clear whether there is a connection between the coefficients $K_{\omega \eta} (q,t)$ and the composition Kostka functions $K_{\lambda \mu}(q,t)$ of \cite{K}. For instance if $\omega=\lambda=(2,2)$ and $\eta=\mu=(3,1)$, we get $K_{\omega \eta} (q,t)=qt+q^2 t+q^3 t^2$ while the composition Kostka function is $K_{\lambda \mu}(q,t) =t+qt +q t^2$ (see (12.13) in \cite{K}), which is equal in this case to the usual $(q,t)$-Kostka coefficient.
But if we replace $\eta=(3,1)$ by $\eta=(0,0,3,1)$, we get $K_{\omega \eta} (q,t)=q^4(t+qt +q t^2)$ which is $K_{\lambda \mu}(q,t)$ up to a $q$-power, suggesting that there might be a non-trivial connection between the two types of coefficients. 
\end{remark}
  
\begin{remark}
As discussed above, there is no loss of information in working modulo
$\mathcal L_m$ when $m \geq \ell+|\eta|$.  Working modulo $\mathcal L_m$ also has the advantage of allowing to
compute the coefficients $K_{\omega \eta} (q,t)$ without having to do the full expansion in the $m$-symmetric Schur function basis. Indeed, we can first
expand $ \wp( J_\eta(x;q,t))$ in the $k_\Omega(x;t)$ basis and then use the relation
\begin{equation} 
k_{\Omega} (x;t)= \sum_{\pmb a \in \mathbb Z_{\geq 0}^m } D_{(\pmb a;\emptyset) \Omega}(t) \, H_{\pmb a}(x;t)  \mod  \mathcal L_m
\end{equation}
which is obtained from  \eqref{eqDs} by 
simply killing the terms involving $m$-symmetric Schur functions indexed by $m$-partitions whose symmetric part is not empty (observe that Proposition~\ref{propoHspecial} allows us to then change the remaining  $s_{(\pmb a;\emptyset)}(x;t)$'s into $ H_{\pmb a}(x;t)$'s).
\end{remark}

\section{Conjectures on the coefficients $K_{\Omega \Lambda}(q,t)$} \label{secButler}
It does  not appear that a perfect analog of Proposition~\ref{propoKostka1} exists when $m>0$.  We can however formulate the following conjecture which states 
that a certain portion of the coefficient $K_{\Omega  \Lambda}(q,t)$ can be recovered from coefficients of lower degree and larger $m$.
\begin{conjecture} \label{conjecf1}
  For any entry $\lambda_j$ in $\lambda$, we have that
$$
 t^{\# \{ k \, | \,  \lambda_j \leq a_k \}}  \left( \begin{array}{ccccc}
 a_1 & \cdots & a_m & \vline &\lambda \\
 b_1 & \cdots & b_m  & \vline & \mu
\end{array}
 \right)-  \sum_{\mu_i} t^{\# \{ k \, | \,  \mu_i \leq b_k \}}
 \left( \begin{array}{cccccc}
  \lambda_j-1 &a_1 & \cdots & a_m & \vline &\lambda \setminus \lambda_j \\
  \mu_i -1  &  b_1 & \cdots & b_m & \vline & \mu \setminus \mu_i
\end{array}
\right)
\in \mathbb N[q,t]
$$
where the sum is over all distinct entries $\mu_i$ in $\mu$.
\end{conjecture}  

The next conjecture tells us how Theorem~\ref{propokostka}(4) can be somehow extended when $a_{i} \neq a_{i+1}$.
\begin{conjecture}\label{conjecf2}
    For a fixed 
  $i<m$, let  $\tilde \Omega =s_i \Omega$ and
  $$
  A_i=\frac{\varepsilon^{(i)}_\Lambda(q,t)}{\varepsilon^{(i+1)}_\Lambda(q,t)}= q^{a_i-a_{i+1}}t^{r_\Lambda(i+1)-r_\Lambda(i)}
  $$
Then, if $b_i > b_{i+1}$, we have that
  \begin{equation}
    \frac{A_it^{-1} K_{\Omega \Lambda} (q,t)- K_{\tilde \Omega \Lambda}(q,t)}{A_it^{-1}-1} \in \mathbb N[q,t] \quad {\rm and} \quad
    \frac{K_{\Omega \Lambda} (q,t)- K_{\tilde \Omega \Lambda}(q,t)}{1-A_it^{-1}} \in \mathbb N[q,t]
  \end{equation}
Note that if $b_i< b_{i+1}$, we only need to switch $\Omega$ and $\tilde \Omega$ in the relations.
\end{conjecture}  
Assuming that Conjecture~\ref{mainconjec} holds, the rule essentially states that in order to obtain $K_{\tilde \Omega  \Lambda}(q,t)$ from $K_{\Omega \Lambda}(q,t)$,
one simply needs to multiply each monomial in $K_{\Omega  \Lambda}(q,t)$ by either $A_i t^{-1}$ or $1$.  Equivalently, in a given combinatorial interpretation using standard tableaux of the coefficients $K_{\Omega \Lambda}(q,t)$, the statistic associated to a given standard tableau would either gain $A_i t^{-1}$ or remain invariant when going from $\Lambda$ to $\tilde \Lambda$.  This phenomenon is very reminiscent of  
Butler's rule on $(q,t)$-Kostka coefficients which relates the coefficients
$K_{\mu \lambda}(q,t)$ and $K_{\mu \nu}(q,t)$, where $\nu$ is a partition that can be obtained from $\lambda$ by moving exactly one box from a given position to another.

The next conjecture is similar to the previous one. It relates this time the entries $ K_{\Omega \Lambda} (q,t)$ and $K_{\tilde \Omega \tilde \Lambda}(q,t)$.
\begin{conjecture} \label{conjecf3}
  For a fixed 
  $i<m$, let  $\tilde \Lambda =s_i \Lambda$ and $\tilde \Omega =s_i \Omega$.

  \noindent (I)  If $a_i>a_{i+1}$ and $b_i > b_{i+1}$, we have that
  \begin{equation}
    \frac{t K_{\Omega \Lambda} (q,t)- K_{\tilde\Omega \tilde\Lambda}(q,t)}{t-1} \in \mathbb N[q,t] \quad {\rm and} \quad
    \frac{K_{\Omega \Lambda} (q,t)- K_{\tilde\Omega \tilde\Lambda}(q,t)}{1-t} \in \mathbb N[q,t]
  \end{equation}

\noindent (II)  If $a_i>a_{i+1}$ and $b_i < b_{i+1}$, we have that
  \begin{equation}
    \frac{t^2 K_{\Omega \Lambda} (q,t)- K_{\tilde \Omega \tilde \Lambda}(q,t)}{t-1} \in \mathbb N[q,t] \quad {\rm and} \quad
    \frac{tK_{\Omega \Lambda} (q,t)- K_{\tilde \Omega \tilde\Lambda}(q,t)}{1-t} \in \mathbb N[q,t]
  \end{equation}
Note again that if $a_i<a_{i+1}$, we simply need to interchange the roles of $\Lambda$ and $\tilde \Lambda$ in the relations. 
\end{conjecture}
\begin{acknow} We thank P. Mathieu, J. Morse and L. Pena for their useful comments. We also thank D. Orr for letting us know about the work contained in  \cite{BG,BG2}, and we  are grateful to B. Goodberry for confirming that the integral form in \cite{BG} was the same as ours.
   \end{acknow}  
\begin{appendix} 
\section{Some additional properties of the $m$-Symmetric Macdonald polynomials}\label{secAppendix}
We prove in this appendix a few properties of the $m$-Symmetric Macdonald polynomials whose proofs were too technical to include in the main presentation.
From now on, we will  always work with a finite number $N$ of variables.  As such  $x$ will denote the variables $(x_1,\dots,x_N)$. Recall that
$ x_{(i)}$ stands for $(x_1,\dots,x_{i-1},x_{i+1},\dots,x_N)$.

The following lemma on non-symmetric Macdonald polynomials is probably already known.  The version given in \cite{Camilo} had stronger conditions but the proof is the same. We reproduce it for completeness. 
\begin{lemma} \label{Camilo}
Let $\eta=(\eta_1,\ldots,\eta_N)$ be a composition such that $\eta_i=0$ and $\eta_j\neq 0$ whenever $j > i$. Then 
\begin{equation}
  E_{\eta}(x)\Big|_{x_i=0}=E_{\eta_{(i)}} (x_{(i)})
\end{equation}
and
\begin{equation}
E_{\eta}(x_1,\ldots,x_N)\Big|_{x_j=0} =0 \quad {\rm if ~} j>i.
\end{equation}
where $\eta_{(i)}=(\eta_1,\ldots,\eta_{i-1}, \eta_{i+1},\ldots,\eta_N)$
stands for the composition $\eta$ without its $i$-th entry.
\end{lemma}

\begin{proof} We proceed by induction. We know from \eqref{property1}
that 
\begin{equation}
  E_{\eta}(x)\Big|_{x_N=0}=E_{\eta_{(N)}} (x_{(N)}) \qquad {\rm when~} i=N
\end{equation}
while 
\begin{equation}
E_{\eta}(x)\Big|_{x_{N}=0}=0 \qquad {\rm when~} i<N {\rm~and~}j=N
\end{equation}
Now,  suppose by induction that
\begin{equation} \label{toprove1}
  E_{\eta}(x)\Big|_{x_{k}=0}=E_{\eta_{(k)}} (x_{(k)}) \qquad {\rm when~} i=k
\end{equation}
while 
\begin{equation} \label{toprove2}
E_{\eta}(x)\Big|_{x_{k}=0}=0 \qquad {\rm when~} i<k {\rm ~and~} j=k
\end{equation}
We thus  need to prove that \eqref{toprove1} and \eqref{toprove2}
still hold when $k$ is replaced by $k-1$.   We first consider the case \eqref{toprove1}. Suppose that $i=k-1$.  Since $\eta_{k-1}=0$ and $\eta_k \neq 0$ by hypothesis, we have from \eqref{property2} that   
\begin{equation} \label{eqrhs1}
  T_{k-1} E_\eta (x) = c_\eta \, E_\eta (x)+tE_{s_{k-1} \eta}(x),
\end{equation}
where $c_\eta$ is an irrelevant constant.  By definition of $T_{k-1}$, we also obtain
\begin{equation} \label{eqrhs2}
  T_{k-1} E_{\eta}(x)=tE_\eta(x) +\frac{t x_{k-1}-x_k}{x_{k-1}-x_k} \bigl(E_{\eta}(x_1,\ldots,x_{k-2},x_k,x_{k-1},x_{k+1},\ldots,x_N)-E_{\eta}(x) \bigr).
\end{equation}
Since $i=k-1$, we have by hypothesis that $E_{\eta}(x)\big|_{x_k=0}=0$
and that
\begin{equation}
  E_{s_{k-1} \eta}(x)\Big|_{x_k=0}=E_{\eta_{(k-1)}}(x_{(k)}).
\end{equation}
Equating the r.h.s. of 
\eqref{eqrhs1} and \eqref{eqrhs2}, we thus get after letting $x_k=0$ that
\begin{equation}
t  E_{\eta}(x_1,\ldots,x_{k-2},0,x_{k-1},x_{k+1},\ldots,x_N) = t E_{\eta_{(k-1)}}(x_{(k)})
\end{equation}  
which, after the change of variables $x_{k-1} \mapsto x_k$, is equivalent to
\begin{equation}
 E_{\eta}(x)\Big|_{x_{k-1}=0} =  E_{\eta_{(k-1)}}(x_{(k-1)})\, .
\end{equation}  
Hence \eqref{toprove1} holds when $k$ is replaced by $k-1$.

We now prove that \eqref{toprove2} holds when $k$ is replaced by $k-1$.  Since $i<j=k-1$, we have by hypothesis that $\eta_{k-1} \neq 0$ and $\eta_k \neq 0$.  We thus obtain by induction that
\begin{equation}
  E_{\eta}(x)\Big|_{x_k=0}=0 \qquad {\rm and} \qquad
  E_{s_{k-1}\eta}(x)\Big|_{x_k=0}=0 .
\end{equation}
Therefore, after again equating the r.h.s. of \eqref{eqrhs1} and \eqref{eqrhs2} and letting $x_k=0$, we get
\begin{equation}
t  E_{\eta}(x_1,\ldots,x_{k-2},0,x_{k-1},x_{k+1},\ldots,x_N)=0\, .
\end{equation}  
But, after 
the change of variables $x_{k-1} \mapsto x_k$, this is amounts to $E_{\eta}(x)\big|_{x_{k-1}=0}=0$.
\end{proof}

\begin{lemma} \label{lemmasymrel}
If $ m+1<i<N$ and $\eta_i > \eta_{i+1}$, then
\begin{equation} \label{symrel}
  \mathcal S^t_{m+1,N} E_\eta =  \frac{(1-tA_i)}{t(1-A_i)}  \mathcal S^t_{m+1,N}  E_{s_i \eta}
\end{equation}
where $A_i=\bar \eta_i/\bar \eta_{i+1}$.
\end{lemma}
\begin{proof} For $i>m+1$, we have $\mathcal S^t_{m+1} T_i = t \mathcal S^t_{m+1}$(this is similar to \eqref{TionSymt}). Hence, from \eqref{property2}, we get
\begin{equation} \label{eqS}
t \mathcal S^t_{m+1,N}  E_\eta =\mathcal S^t_{m+1,N} T_i E_\eta= \mathcal S^t_{m+1,N} \left( \frac{(t-1)}{(1-A_i^{-1})} E_\eta + \frac{(1-tA_i)(1-t^{-1}A_i)}{(1-A_i)^2} E_{s_i \eta} \right) 
\end{equation}
Comparing the l.h.s and the r.h.s of \eqref{eqS}, we then obtain
$$
\frac{t(1-t^{-1}A_i)}{(1-A_i)} \mathcal S^t_{m+1,N} E_\eta =  \frac{(1-tA_i)(1-t^{-1}A_i)}{(1-A_i)^2} \mathcal S^t_{m+1,N} E_{s_i \eta} 
$$
which readily implies our claim.
\end{proof}

\begin{lemma}  
  Suppose that $\eta_{m+2}=\eta_{m+3}=\cdots = \eta_{m+n+1}=0$ (with $n \geq 1$) are the only zero entries in $\eta=(\eta_1,\dots,\eta_N)$ to the right of $\eta_{m+1}>0$.
  Then
\begin{equation} \label{bigresult}
\mathcal S^t_{m+2,N} E_{\eta}(x) \Big|_{x_{m+1}=0}= \frac{A_{m+1} (1-t^n)}{t (1-t^{n-1}A_{m+1})}\mathcal S^t_{m+2,N} E_{\eta_{(m+2)}}(x_{(m+1)})
\end{equation}
where $A_{m+1}=\bar \eta_{m+1}/\bar \eta_{m+2}$.
\end{lemma}
\begin{proof} 
From \eqref{property2}, we get
$$
T_{m+1} E_{\eta} \Big |_{x_{m+1}=0} = \frac{(t-1)}{(1-A_{m+1}^{-1})} E_\eta \Big |_{x_{m+1}=0} + \frac{(1-tA_{m+1})(1-t^{-1}A_{m+1})}{(1-A_{m+1})^2} E_{s_{m+1} \eta} \Big |_{x_{m+1}=0}
$$
But we also have directly using \eqref{eqTi} that
$$
T_{m+1} E_{\eta} \Big |_{x_{m+1}=0} =(t-1)   E_\eta \Big |_{x_{m+1}=0}  +  E_{\eta} \Big |_{(x_{m+1},x_{m+2}) \mapsto (x_{m+2},0)} 
$$
The previous two equations then imply that
\begin{equation} \label{casegen}
\frac{(1-t)}{(1-A_{m+1})}  E_\eta \Big |_{x_{m+1}=0}=  E_{ \eta} \Big |_{(x_{m+1},x_{m+2}) \mapsto (x_{m+2},0)}   -\frac{(1-tA_{m+1})(1-t^{-1}A_{m+1})}{(1-A_{m+1})^2} E_{s_{m+1} \eta} \Big |_{x_{m+1}=0} 
\end{equation}
First, suppose that $n=1$ (that is, that there is only one zero entry to the right of $\eta_{m+1}$).  From Lemma~\ref{Camilo}, we have
$$
E_{ \eta} \Big |_{(x_{m+1},x_{m+2}) \mapsto (x_{m+2},0)} = E_{\eta_{(m+2)}}(x_{(m+1)}) \qquad {\rm and } \qquad E_{s_{m+1} \eta} \Big |_{x_{m+1}=0} = E_{\eta_{(m+2)}}(x_{(m+1)})
$$
Using those equalities in \eqref{casegen} gives after some simplifications
\begin{equation} 
  \frac{(1-t)}{(1-A_{m+1})}  E_\eta \Big |_{x_{m+1}=0}=
  \frac{A_{m+1}(1-t)^2}{t(1-A_{m+1})^2} E_{\eta_{(m+2)}} (x_{(m+1)}) 
\end{equation}
and it is then immediate that \eqref{bigresult} holds when $n=1$.  When $n>1$, we need to work with   $\mathcal S^t_{m+1}$ on the left from the start to prove our claim. Using \eqref{symrel} again and again (and realizing that each time $\eta_{m+1}$ is switched with the zero to its right, the corresponding factor $A_{m+1}$ is multiplied by $t$), we have
$$
\mathcal S^t_{m+2,N} E_{s_{m+1} \eta} \Big |_{x_{m+1}=0}   =  \frac{(1-t^2A_{m+1}) \cdots (1-t^n A_{m+1})}{t^{n-1}(1-tA_{m+1}) \cdots (1-t^{n-1}A_{m+1})}  \mathcal S^t_{m+2,N}  E_{s_{m+n} \cdots s_{m+1}  \eta} \Big |_{x_{m+1}=0}
$$
 Lemma~\ref{Camilo} gives
$$
 E_{s_{m+n} \cdots s_{m+1}  \eta} \Big |_{x_{m+n}=0} =E_{(s_{m+n-1} \cdots s_{m+1}  \eta)_{(m+n)}}(x_{(m+n)})=
 E_{s_{m+n-1} \cdots s_{m+1}  \eta_{(m+2)}}(x_{(m+n)})
$$
which implies that
$$
E_{s_{m+n} \cdots s_{m+1}  \eta} \Big |_{x_{m+1}=0} = E_{s_{m+n-1} \cdots s_{m+1}  \eta_{(m+2)}} (x_{(m+1)})
$$
since $ E_{s_{m+n} \cdots s_{m+1}  \eta}$ is symmetric in the variables $x_{m+1}, \dots,x_{m+n}$ by Remark~\ref{remarksym}. 
 We thus obtain
$$
\mathcal S^t_{m+2,N} E_{s_{m+1} \eta} \Big |_{x_{m+1}=0}   =
 \frac{(1-t^2A_{m+1}) \cdots (1-t^n A_{m+1})}{t^{n-1}(1-tA_{m+1}) \cdots (1-t^{n-1}A_{m+1})}
  \mathcal S^t_{m+2,N}  E_{s_{m+n-1} \cdots s_{m+1}  \eta_{(m+2)}}(x_{(m+1)})  
$$
Now using \eqref{symrel} in the other direction (with a shift in $t$ due to the loss of a zero to the left), we get
$$
\mathcal S^t_{m+2,N} E_{s_{m+1} \eta} \Big |_{x_{m+1}=0}   =  \frac{(1-A_{m+1})(1-t^n A_{m+1})}{(1-tA_{m+1})(1-t^{n-1}A_{m+1})} \mathcal S^t_{m+2,N}  E_{\eta_{(m+2)}}(x_{(m+1)})  
$$
On the other hand,  using again Lemma~\ref{Camilo} yields 
$$
E_{ \eta} \Big |_{x_{m+n+1}=0} =E_{\eta_{(m+n+1)}}(x_{(m+n+1)})
$$
which implies this time that
$$
E_{ \eta} \Big |_{x_{m+2}=0} =E_{\eta_{(m+2)}}(x_{(m+2)})
$$
given that $E_{\eta}$ is symmetric in the variables  $x_{m+2}$ up to $x_{m+n+1}$
and that $\eta_{(m+n+1)}=\eta_{(m+2)}$.
Hence
$$
E_{ \eta} \Big |_{(x_{m+1},x_{m+2}) \mapsto (x_{m+2},0)} =E_{\eta_{(m+2)}}(x_{(m+2)}) \Big |_{x_{m+1} \mapsto x_{m+2}} =E_{\eta_{(m+2)}}(x_{(m+1)})  
$$
  Equation \eqref{casegen} then becomes
 $$
 \frac{(1-t)}{(1-A_{m+1})} \mathcal S^t_{m+2,N} E_\eta \Big |_{x_{m+1}=0}=
\left(  1   -\frac{(1-t^n A_{m+1})(1-t^{-1}A_{m+1})}{(1-t^{n-1}A_{m+1})(1-A_{m+1})} \right) \mathcal S^t_{m+2,N}  E_{\eta_{(m+2)}}(x_{(m+1)})  
 $$
which, after some straightforward manipulations, yields 
$$
 \frac{(1-t)}{(1-A_{m+1})}  \mathcal S^t_{m+2,N} E_\eta \Big |_{x_{m+1}=0}=
\frac{A_{m+1}(1-t^n )(1-t)}{t (1-t^{n-1}A_{m+1})(1-A_{m+1})} \mathcal S^t_{m+2,N}  E_{\eta_{(m+2)}}(x_{(m+1)})  
 $$
The claim \eqref{bigresult} then follows after performing the obvious cancellations. 
\end{proof}

\begin{proposition} \label{propMacdo00}
  For $\Lambda=(a_1,\dots,a_m ; \lambda )$ and  $ \Lambda^{0} =(a_1,\dots,a_m,0; \lambda)$, we have that
$$  
  c_{\Lambda^{0}} (q,t) P_{\Lambda^{0}}(x;q,t) \Big |_{x_{m+1}=0} =    c_{\Lambda} (q,t) P_{ \Lambda} (x_{(m+1)};q,t)
  $$
  where the number of variables $N$ in $x=(x_1,\dots,x_N)$ is such that
  $N>  \ell(\Lambda)$.
  \end{proposition}
\begin{proof}
  The extra circle $m+1$ at the bottom of  $\Lambda^0$ does not affect any of the hook-lengths, from which we get that $ c_{\Lambda^{0}} (q,t)= c_{\Lambda} (q,t)$.
   It is also easy to see that $ u_{\Lambda^{0},N} (t)=  u_{\Lambda,N-1} (t)$.
From the definition of the $m$-symmetric Macdonald polynomials, we thus have left to prove that
$$
\mathcal S_{m+2,N}^t E_{\eta_{\Lambda^0,N}} (x;q,t) \Big |_{x_{m+1=0}}  =
  \mathcal S_{m+2,N}^t E_{\eta_{\Lambda,N-1}} (x_{(m+1)};q,t)
  $$
  Suppose that in $\eta_{\Lambda^0,N}$ there are zeroes in entries $m+1$ up to $m+n$.  From  Lemma~\ref{Camilo}, we obtain that 
  $$
 E_{\eta_{\Lambda^0,N}} (x;q,t) \Big |_{x_{m+n=0}}=  E_{\eta_{\Lambda,N-1}} (x_{(m+n)};q,t)
 $$
which implies that 
 $$
 E_{\eta_{\Lambda^0,N}} (x;q,t) \Big |_{x_{m+1=0}}=  E_{\eta_{\Lambda,N-1}} (x_{(m+1)};q,t)
 $$
since $E_{\eta_{\Lambda^0,N}} (x;q,t)$ is symmetric in the 
variables $x_{m+1},\dots,x_{m+n}$ by Remark~\ref{remarksym}.
 The proposition then follows given that $\mathcal S_{m+2,N}^t$ does not depend on $x_{m+1}$.  
\end{proof}  

\begin{proposition} \label{propMacdo0}
  For $\Lambda=(a_1,\dots,a_m ; \lambda )$ and  $ \Lambda^{\diamond} =(a_1,\dots,a_m,\lambda_j; \lambda \setminus \lambda_j)$, we have that
$$  
  c_{\Lambda^{\diamond}} (q,t) P_{\Lambda^{\diamond}}(x;q,t) \Big |_{x_{m+1}=0} = q^{\lambda_j} t^{\# \{ i \, | \, a_i < \lambda_j\} }   c_{\Lambda} (q,t) P_{ \Lambda} (x_{(m+1)};q,t)
  $$
  where the number of variables $N$ in $x=(x_1,\dots,x_N)$ is such that
  $N>  \ell(\Lambda)$.
  \end{proposition}
\begin{proof}

 We have that
\begin{equation} \label{eqc}
\frac{ c_{\Lambda^{\diamond}} (q,t)}{ c_{\Lambda} (q,t)}= \prod_{s \in \lambda_j}\frac{1-q^{a_{ \Lambda}(s)+1} t^{l_{ \Lambda}(s)+1+\circ_{\Lambda}(s)}}{1-q^{a_{  \Lambda}(s)} t^{l_{ \Lambda}(s)+1}}
\end{equation}
where the product is over all boxes $s$ in the row corresponding to $\lambda_j$ in the diagram of $\Lambda$  (the highest row of size $\lambda_j$ that does not end with a circle in the diagram of $\Lambda$), and where $\circ_{\Lambda}(s)$ is equal to the number of circles in the column of $s$ (the new circle $m+1$ at the end of the row of size $\lambda_j$ will  always be larger than any circle in the column of $s$).

We also have that the normalization $u_{\Lambda,N}$ in \eqref{normalization} is such that
 \begin{equation} \label{eqNorm}
\frac{ u_{\Lambda,N-1} (t)}{ u_{\Lambda^{\diamond},N} (t)}= \frac{(1-t^{-n_\lambda(\lambda_j)})}{(1-t^{-n})} = \frac{(1-t^{n_\lambda(\lambda_j)})}{t^{n_\lambda(\lambda_j)-n}(1-t^{n})}
 \end{equation}

 Let $\nu=\lambda \setminus \lambda_j$ and let $\hat \lambda$ be the partition obtained from $\lambda$ by removing all the parts larger than $\lambda_j$.
Using \eqref{bigresult}, we obtain that
$$
\mathcal S_{m+2,N}^t E_{\eta_{\Lambda^\diamond,N}} (x;q,t) \Big |_{x_{m+1=0}}  = \frac{A_{m+1}(1-t^n)}{t(1-t^{n-1}A_{m+1})}
 \mathcal S_{m+2,N}^t E_{a_1,\dots,a_m,\lambda_j,0^{n-1},\nu_{\ell(\nu)},\dots,\nu_1 } (x_{(m+1)};q,t)
 $$
  We now need to use 
\eqref{symrel} again and again to move $\lambda_j$ to the right
to obtain the composition
$$(a_1,\dots,a_m,0^{n-1},\lambda_{\lambda(\nu)},\dots,\lambda_1) =\eta_{\Lambda,N-1}$$
Let $f_\Lambda(q,t)$ be the constant that appears in this process, that is, 
 $$
\mathcal S_{m+2,N}^t E_{a_1,\dots,a_m,\lambda_1,0^{n-1},\nu_{\ell(\nu)},\dots,\nu_1 }(x_{(m+1)};q,t)=
f_{\Lambda}(q,t)
\mathcal S_{m+2,N}^t E_{\eta_{\Lambda,N-1}}(x_{(m+1)};q,t)
 $$
Taking \eqref{eqc} and \eqref{eqNorm} into consideration and using
$A_{m+1}=q^{\lambda_j} t^{\# \{ i \, | \, a_i < \lambda_j\} + \ell(\hat \lambda)}$, the proposition will follow from the definition of the $m$-symmetric Macdonald polynomials if we can
prove that
  $$
\left(\prod_{s \in \lambda_j}\frac{1-q^{a_{ \Lambda}(s)+1} t^{l_{ \Lambda}(s)+1+\circ_{\Lambda}(s)}}{1-q^{a_{  \Lambda}(s)} t^{l_{ \Lambda}(s)+1}} \right) \frac{q^{\lambda_j}t^{\# +\ell(\hat \lambda)-n_\lambda(\lambda_j)+n-1}(1-t^{n_\lambda(\lambda_j)})}{(1-q^{\lambda_j} t^{\#  + \ell(\hat \lambda)+n-1})} f_\Lambda (q,t)=q^{\lambda_1} t^{\# }
$$
where we use the shorthand $\#$ for $\# \{ i \, | \, a_i < \lambda_j\}$.
This amounts to showing that
\begin{equation} \label{product}
f_\Lambda (q,t) = \left(\prod_{s \in \lambda_j}\frac{1-q^{a_{  \Lambda}(s)} t^{l_{ \Lambda}(s)+1}}{1-q^{a_{ \Lambda}(s)+1} t^{l_{ \Lambda}(s)+1+\circ_{\Lambda}(s)}} \right) \frac{(1-q^{\lambda_j} t^{\# + \ell(\hat \lambda)+n-1})} {t^{\ell(\hat \lambda)-n_\lambda(\lambda_j)+n-1}(1-t^{n_\lambda(\lambda_j)})}
\end{equation}
We proceed by induction on the number of exchanges that need to be done in order to move $\lambda_j$ to the right.

If there are no exchanges to do, we need to show that $f_\Lambda(q,t)=1$.
In this case, we have $n=1$ and $\lambda=(\lambda_j^{n_\lambda(\lambda_j)})$. Moreover, in the diagram of $\Lambda$, all the rows that are shorter than $\lambda_j$ end with a circle.  Hence, if there is a cell  $s'$  immediately to the right of $s$ in the row of $\lambda_j$, we get
$$
\frac{1-q^{a_{  \Lambda}(s)} t^{l_{ \Lambda}(s)+1}}{1-q^{a_{ \Lambda}(s')+1} t^{l_{ \Lambda}(s')+1+\circ_{\Lambda}(s)}}=1
$$
The first product in $f_\Lambda(q,t)$ thus simplifies to
$$
\prod_{s \in \lambda_j}\frac{1-q^{a_{  \Lambda}(s)} t^{l_{ \Lambda}(s)+1}}{1-q^{a_{ \Lambda}(s)+1} t^{l_{ \Lambda}(s)+1+\circ_{\Lambda}(s)}} = \frac{(1-t^{n_\lambda(\lambda_j)})}{(1-q^{\lambda_j}t^{\# +n_\lambda(\lambda_j)})} 
$$
From  \eqref{product}, this yields that $f_\Lambda(q,t)=1$
since $n=1$ and $n_\lambda(\lambda_j)=\ell(\hat \lambda)$.

We now consider the general case.  First suppose that $n=1$ (the case in which there are no more $0$'s with which to do do an exchange).  The first exchange that needs to be done is with a circle in a row of size $\lambda_\ell=\lambda_{\ell(\lambda)}$. From \eqref{symrel},
it gives rise to the term
\begin{equation} \label{eqavant}
\frac{1-q^{\lambda_j-\lambda_\ell}t^{l_{\Lambda}(s')+2+\circ_\Lambda(s')}}{t(1-q^{\lambda_j-\lambda_\ell}t^{l_{\Lambda}(s')+1+\circ_\Lambda(s')})}
\end{equation}
where $s'$ is the cell in the row of $\lambda_j$ and column $\lambda_j-\lambda_\ell+1$. The remaining exchanges are equal by induction to $f_\Omega(q,t)$, where
$\Omega$ is $\Lambda$ with the row corresponding to $\lambda_\ell$ now having a circle.  Considering that $n=1$, we obtain
\begin{equation} \label{eqfo}
f_\Omega(q,t)=\left(\prod_{s \in \lambda_1}\frac{1-q^{a_{  \Omega}(s)} t^{l_{ \Omega}(s)+1}}{1-q^{a_{ \Omega}(s)+1} t^{l_{ \Omega}(s)+1+\circ_{\Omega}(s)}} \right) \frac{(1-q^{\lambda_j} t^{\# + \ell(\hat \lambda)})} {t^{\ell(\hat \lambda)-1-n_\lambda(\lambda_j)}(1-t^{n_\lambda(\lambda_j)})}
\end{equation}
where we used the fact that, when compared to $\Lambda$, $\# $ increased by 1 while $\ell(\lambda)$ decreased by one.  We thus have to show that the product of \eqref{eqavant} and \eqref{eqfo} gives \eqref{product} in the case $n=1$.  After performing easy cancellations, this amounts to showing that
$$
\frac{1-q^{\lambda_j-\lambda_\ell}t^{l_{\Lambda}(s')+2+\circ_\Lambda(s')}}{1-q^{\lambda_j-\lambda_\ell}t^{l_{\Lambda}(s')+1+\circ_\Lambda(s')}}\left(\prod_{s \in \lambda_1}\frac{1-q^{a_{  \Omega}(s)} t^{l_{ \Omega}(s)+1}}{1-q^{a_{ \Omega}(s)+1} t^{l_{ \Omega}(s)+1+\circ_{\Omega}(s)}} \right) 
 = \left(\prod_{s \in \lambda_1}\frac{1-q^{a_{  \Lambda}(s)} t^{l_{ \Lambda}(s)+1}}{1-q^{a_{ \Lambda}(s)+1} t^{l_{ \Lambda}(s)+1+\circ_{\Lambda}(s)}} \right) 
$$
 The only difference in the products over the cells $s$ is when $s=s'$, where
$s'$ lies in the column of the extra circle in $\Omega$ (in column $\lambda_j-\lambda_\ell +1$).  
Hence, we need to show that
$$
\frac{1-q^{\lambda_j-\lambda_\ell}t^{l_{\Lambda}(s')+2+\circ_\Lambda(s')}}{1-q^{\lambda_j-\lambda_\ell}t^{l_{\Lambda}(s')+1+\circ_\Lambda(s')}}
\frac{1-q^{a_{  \Omega}(s')} t^{l_{ \Omega}(s')+1}}{1-q^{a_{ \Omega}(s')+1} t^{l_{ \Omega}(s')+1+\circ_{\Omega}(s')}}
 = \frac{1-q^{a_{  \Lambda}(s')} t^{l_{ \Lambda}(s')+1}}{1-q^{a_{ \Lambda}(s')+1} t^{l_{ \Lambda}(s')+1+\circ_{\Lambda}(s')}} 
$$
 Using $a_{\Lambda}(s')=a_{\Omega}(s')=\lambda_j-\lambda_{\ell}-1$, $l_{\Lambda}(s')=l_{\Omega}(s')$,
 and $\circ_{\Omega}(s')=\circ_{\Lambda}(s')+1$, the result is seen to hold.

 Finally, in the case $n>1$, the proof is exactly as in the previous case except that $s'$ is now in the first column, $\lambda_\ell$ is replaced by 0, and the quantities $(1-q^{\lambda_j} t^{\# + \ell(\lambda)})$ and $t^{\ell(\hat \lambda)-n_\lambda(\lambda_j)}$ are  replaced by $(1-q^{\lambda_j} t^{\# + \ell(\hat \lambda)+n-1})$ and $t^{\ell(\hat \lambda)-n_\lambda(\lambda_j)+n-1}$ respectively.

\end{proof}  

We will finish this appendix with a proof Theorem~\ref{propom1}.
Recall that its main statement says that
for any $N$ and any $m \geq 1$, we have that
  $$
\Psi_N J_\Lambda (x_1,\dots,x_N;q,t)= t^{-\# \{ 2 \leq j \leq m \, | \, a_j \leq  a_1 \}} 
J_{\Lambda^\square} (x_1,\dots,x_N;q,t)
$$
where $\Lambda^\square =\bigl((a_2,\dots,a_{m}); \lambda \cup (a_1+1)\bigr)$.
\begin{proof}[Proof of Theorem~\ref{propom1}]
  We first establish that
\begin{equation} \label{eqPsi2}
\Psi_N \mathcal S_{m+1,N}= (1-t)\mathcal S_{m,N} \Phi_q
\end{equation}
for $m \geq 1$. 
From the braid relations, it is easy to see that
$$
(T_{N-1} \cdots T_1) T_i = T_{i-1} (T_{N-1} \cdots T_1)  
$$
for all $i=2,\dots,N-1$. Hence, for $m\geq 1$, we obtain
$$
\Phi_q \mathcal S_{m+1,N}= \mathcal S_{m,N-1} \Phi_q 
$$
where $\Phi_q$ is defined in \eqref{eqPhi}. It is then immediate from \eqref{extraN} that \eqref{eqPsi2} holds.

Now, given that by definition
  $$
J_\Lambda = \frac{c_\Lambda(q,t)}{u_{\Lambda,N}(t)} \mathcal S_{m+1,N} E_{\Lambda,N} 
$$
we have to show that
\begin{equation} \label{toShow}
\Psi_N  \mathcal S_{m+1,N} E_{\eta_{\Lambda,N}}  = t^{-\# \{ 2 \leq j \leq m \, | \, a_j \leq  a_1 \}}    \frac{c_\Lambda^\square(q,t)}{c_{\Lambda}(q,t)}  \frac{u_{\Lambda,N}(t)}{u_{\Lambda^\square,N}(t)} \mathcal S_{m,N} E_{\eta_{\Lambda^\square,N}} 
\end{equation}
When the circle indexed by a 1 in the diagram of $\Lambda$ becomes a square, it will only change the contribution in $c_\Lambda(q,t)$ of the cells in the column of the circle indexed by a 1 that lie in rows that do not end with a circle (otherwise the cell indexed by a 1 already contributes to the leg-length).  Hence, also considering the change in the symmetric rows of length $a_1+1$, we obtain
\begin{equation} \label{eqC}
 \frac{c_\Lambda^\square(q,t)}{c_{\Lambda}(q,t)} = \left( \prod_{s \in {  \rm col}(1)} \frac{1-q^{a_\Lambda(s)}t^{l_\Lambda(s)+2}}{1-q^{a_\Lambda(s)}t^{l_\Lambda(s)+1}} \right) (1-t^{n_\lambda(a_1+1)+1})
\end{equation}
 where the product is over the cells in the column of the circle indexed by 1 that lie in  rows that are larger than $a_1+1$ and that do not end with a circle.
 We also have
\begin{equation} \label{equ}
 \frac{u_{\Lambda,N}(t)}{u_{\Lambda^\square,N}(t)}= \left(\frac{1-t^{-1}}{1-t^{-n_\lambda(a_1+1)-1}}\right) \frac{t^{(N-m)(N-m-1)/2}}{t^{(N-m+1)(N-m)/2}}= \left(\frac{1-t}{1-t^{n_\lambda(a_1+1)+1}}\right) t^{n_\lambda(a_1+1)+m-N}
 \end{equation}
On the other hand, we obtain from \eqref{eqPhi} and \eqref{eqPsi} that
\begin{equation} \label{otherhand}
\Psi_N  \mathcal S_{m+1,N} E_{\eta_{\Lambda,N}} = (1-t) t^{r_\Lambda(1)-N}  \mathcal S_{m,N} E_{\Phi (\eta_{\Lambda,N})}  
\end{equation}
Using Lemma~\ref{lemmasymrel} again and again to reorder the last entry of $\Phi (\eta_{\Lambda,N})$, we  obtain
$$
\mathcal S_{m,N} E_{\Phi (\eta_{\Lambda,N})}  =t^{-\#\{j \, | \lambda_j > a_1+1\}} \left( \prod_{s \in {  \rm col}(1)} \frac{1-q^{a_\Lambda(s)}t^{l_\Lambda(s)+2}}{1-q^{a_\Lambda(s)}t^{l_\Lambda(s)+1}} \right) \mathcal S_{m,N} E_{\eta_{\Lambda^\square,N}}  
$$
where the product is again  over the cells in the column of the circle indexed by 1 that lie in  rows that are larger than $a_1+1$ and that do not end with a circle.
Therefore, \eqref{otherhand} becomes
$$
\Psi_N  \mathcal S_{m+1,N} E_{\eta_{\Lambda,N}}= (1-t) t^{r_\Lambda(1)-N-\#\{j \, | \lambda_j > a_1+1\}}  \left( \prod_{s \in {  \rm col}(1)} \frac{1-q^{a_\Lambda(s)}t^{l_\Lambda(s)+2}}{1-q^{a_\Lambda(s)}t^{l_\Lambda(s)+1}} \right) \mathcal S_{m,N} E_{\eta_{\Lambda^\square,N}}  
$$
Using \eqref{eqC} and \eqref{equ}, we then see from the previous equation that \eqref{toShow} will hold if 
$$
r_\Lambda(1)-N-\#\{j \, | \lambda_j > a_1+1\}= -{\# \{ 2 \leq j \leq m \, | \, a_j \leq  a_1 \}}  +n_\lambda(a_1+1)+m-N
$$
But this easily follows after realizing that $r_\Lambda(1)=\ell(\Lambda)-{\# \{ 2 \leq j \leq m \, | \, a_j \leq  a_1 \}}-{\# \{ j  \, | \, \lambda_j \leq  a_1 \}}$ and that $\ell(\Lambda)=m+{\# \{ j  \, | \, \lambda_j \leq  a_1 \}}+{\# \{ j  \, | \, \lambda_j >  a_1+1 \}}+n_\lambda(a_1+1)$.
\end{proof}

\end{appendix}

\end{document}